\documentclass[intlimits]{cmiae}

\usepackage{amssymb, amsmath, amscd,graphicx}

\renewcommand{\S}{\operatorname{S}}

\def\R{{\mathbb R}}
\def\N{{\mathbb N}}
\def\Z{{\mathbb Z}}

\def\F{\Phi}

\def\a{\alpha}
\def\b{\beta}

\def\eps{\varepsilon}

\def\s{\sigma}
\def\t{\theta}

\def\lam{\lambda}

\def\tcut{t_{\operatorname{cut}}}
\def\tt{t_1}

\newcommand{\MAX}{\operatorname{MAX}\nolimits}

\newcommand{\Exp}{\operatorname{Exp}\nolimits}

\newcommand{\SO}{\operatorname{SO}}

\def\Fb{\F^{\b}}

\newcommand{\pder}[2]{\frac{\partial \, #1}{\partial \, #2} }
\newcommand{\pnder}[3]{\frac{\partial^#3 \, #1}{\partial \, #2^#3}}
\newcommand{\be}[1]{\begin{equation}\label{#1}}
\newcommand{\ee}{\end{equation}}

\newcommand{\eq}[1]{$(\protect\ref{#1})$}

\newcommand{\orr}[2]{\left[\begin{array}{l}{#1}\\[0.1cm]{#2}\end{array
}\right.}
\newcommand{\orrr}[3]{\left[\begin{array}{l}{#1}\\{#2}\\{#3}\end{array
}\right.}

\def\tlam{\widetilde{\lambda}}

\def\bx{\bar{x}}
\def\by{\bar{y}}

\def\bq{\bar{q}}

\def\bu{\bar{u}}

\def\iff{\quad\Leftrightarrow\quad}

\newcommand{\vectortwo}[2]{
\left(\begin{array}{c}
#1\\
#2	
\end{array}\right)
}
\newcommand{\matrixtwo}[4]{
\left(
\begin{array}{cc}
#1&#2\\
#3&#4	
\end{array} \right)
}
\newcommand{\matrixrotation}[1]{\matrixtwo{\cos #1}{-\sin #1}{\sin 
#1}{\cos #1}}

\newcommand{\pdd}[2]{\frac{\partial #1}{\partial #2}}

\newcommand{\sast}{s^{\ast}}
\newcommand{\mast}{m^{\ast}}
\newcommand{\past}{p^{\ast}}

\newcommand{\sign}{\mathrm{sign}}

\newcommand{\barm}{\bar{m}}

\renewcommand\geq{\geqslant}
\renewcommand\leq{\leqslant}

\newcounter{abcd}
\newtheorem{theoremSachkov}{Theorem}[abcd]

\newtheorem{theorem}{Theorem}[section]
\newtheorem{lemma}{Lemma}[section]
\newtheorem{corollary}{Corollary}[section]
\newtheorem{proposition}{Proposition}[section]

\newtheorem{remark}{Remark}[section]

\begin{document}

\title
{Asymptotics of Maxwell time in the plate-ball problem\footnote{Supported by Russian Foundation for Basic Research, Project No.~09-01-00246-а.}}

\author{A.P. Mashtakov, A.Yu. Popov}
\address{Program Systems Institute of RAS}
\email{alexey.mashtakov@gmail.com}

\begin{abstract}
The problem on rolling of a sphere on a plane without slipping or twisting is considered. One should roll the sphere from one contact configuration to another so that the length of the curve traced by the contact point in the plane was the shortest possible. Asymptotics of Maxwell time for rolling of the sphere along small amplitude sinusoids is studied. Two-sided estimate for this asymptotics is obtained. 
\end{abstract}

\maketitle
\section{Introduction}
\label{sec:intro}
For the problem on rolling of a sphere on a plane without slipping or twisting, an optimal control problem is studied. State of the system is described by the contact point of the sphere and the plane and orientation of the sphere in three-dimensional space.
One should roll the sphere from one contact configuration to another so that the length of the curve traced by the contact point in the plane was the shortest possible. The problem has application in robotics: rotation of a solid body in robot's hand. In this work we obtain two-sided estimate of Maxwell time in the plate-ball problem in asymptotic case.

The problem was stated in ~\cite{hammersley} by Hammersley. Then Arthur and Walsh~\cite{arthurs_walsh} proved integrability of Hamiltonian system of PMP in elliptic functions. Jurdjevic in~\cite{jurd_plate-ball, jurd_book} showed that projections of extremal curves $(x(t), y(t))$ are Euler elasticae (see the papers ~\cite{euler,love})). He gave a description of different qualitative types of extremal trajectories, and obtained differential equations for evolution of Euler angles along extremal trajectories. Explicit formulas for the extremals were obtained in the paper~\cite{s2_as}.  

Optimality of extremals is still an open problem nowadays. Short arcs of extremal
trajectories are optimal but long arcs, in general, are not optimal. A point at which an extremal trajectory loses global optimality is called a cut point. A cut point is a conjugate point or a Maxwell point. A Maxwell point is a point in the state space, where an extremal trajectory crosses another one with the same value of cost functional. Yu.~Sachkov began to study cut points in the plate-ball problem (see the paper~\cite{s2r2}). He found continuous and discrete symmetries of the exponential mapping. Then he obtained equations, which define Maxwell points as fixed points of the discrete symmetries, and formulated necessary optimality conditions in terms of Maxwell time (see Theorems~\ref{th:Max1},~\ref{th:Max2} in Section~\ref{sec:known-results}). But the problem of optimality of extremal trajectories is steel open because the equations, which define the Maxwell points, are not solved.  

The papers~\cite{s2_as,masht1} present the asymptotics of extremal trajectories in a neighborhood of the stable equilibrium of a mathematical pendulum (see~\eq{pend}), which appears in the adjoint subsystem of the Hamiltonian system of the Maximum Principle. In this case, the extremal curves on the plane are close to sinusoids of small amplitude.   

This work continues to study the problem of optimality of extremal trajectories. It studies the problem in an asymptotic case, where the formulas defining the extremal trajectories and the Maxwell points are expressed via trigonometric functions. They are simpler than in the general case where the formulas are expressed in elliptic functions. We study the behavior of Maxwell points $\MAX^1$ and $\MAX^2$ in this case and obtain two-sided estimates for the first Maxwell times  $t_1$ and $t_2$, which correspond to the fixed points of the discrete symmetries $\eps^1$ and $\eps^2$ of the exponential mapping.    
\section{Statement of the problem and known results}
\label{sec:known-results}
In this section we formulate the plate-ball problem and recall some known results. Let $(x,y)\in \R^2$ be the contact point of the sphere and the plane. By $q =(q_0,q_1,q_2,q_3)\in \S^3$ denote the unit quaternion (see the paper~\cite{pontryagin_quatern}) representing the rotation of three-dimensional space, which translates the current orientation of the sphere to the initial orientation. The problem on optimal rolling of a unit sphere on a plane is stated as follows: 
\begin{align}
&\dot{Q} = u_1 X_1(Q) + u_2 X_2(Q), \label{dotxy}\\
&X_1(Q) = (1,0,q_2,q_3,-q_0,-q_1)^T,\qquad X_2(Q) = (0,1,-q_1,q_0,q_3,-q_2)^T,\\
&Q = (x,y,q_0,q_1,q_2,q_3) \in M = \R^2 \times \S^3, \qquad u = (u_1, u_2) \in \R^2, \label{sysQu}\\
&Q(0) = Q_0 = (0, 0, 1, 0, 0, 0), \qquad Q(t_1) = Q_1, \label{sysQ0}\\
&l = \int_0^{t_1} \sqrt{u_1^2 + u_2^2} \, dt \to \min. \label{sysl}
\end{align}
Admissible controls are measurable and essentially bounded. Admissible trajectories are Lipschitz. Problem~\eq{dotxy}--\eq{sysl} is a left-invariant sub-Riemannian problem
on the Lie group $M = \R^2 \times \S^3$. The control system is completely
controllable by the Rashevsky--Chow theorem (see the paper~\cite{notes}). Existence of optimal controls follows from Filippov's theorem (see~\cite{notes}). The Maximum Principle is applied to study the optimal controls. In the abnormal case, the sphere rolls along a straight line in the plane $(x,y)$. In the normal case, the subsystem of the Hamiltonian system for the adjoint variables $(\t, c, r, \a)$ satisfies the equations of a mathematical pendulum as follows:
\begin{eqnarray}
 &\dot{\theta} = c,\quad \dot{c} = -r \sin\theta,\quad \dot{\alpha} = \dot{r} = 0.  \label{pend}
\end{eqnarray}
Projections of extremal trajectories to the plane $(x, y)$ are Euler elasticae, i.e. stationary
configurations of an elastic rod on a plane with fixed end points and fixed
tangents at these points (see the paper~\cite{jurd_plate-ball}).

In the paper~\cite{s2r2} the author describes continuous and discrete symmetries of the exponential mapping in the plate-ball problem
\begin{align*}
\Exp \ &: \ (\lam,t) \mapsto Q_t, \quad (\lam,t) \in N = C \times 
\R_+, 
\quad Q_t \in M = \R^2 \times \SO(3),\\
C  &= \{\lam \in T_{Q_0}^* M \mid H(\lam) = 1/2\} = \{ (\t, c, \a, r) \mid \t \in S^1, \quad  c \in \R, \quad  r \geq 
0, \quad  \a \in S^1\}. 
\end{align*}
The continuous symmetries $\{\F^{\b} \mid \b \in S^1 \}$ are rotations by the angle $\b$ in the plane $(x,y)$. The discrete symmetries $\eps^1$, $\eps^2$, $\eps^3$ are reflections of the trajectories of pendulum~\eq{pend} in the coordinate axes $\{c = 0\}$, $\{\t = 0\}$, and in the origin $(\t,c) = (0,0)$ respectively. The action of the symmetries in the preimage $N$ and in the image $M$ of the exponential mapping is defined in~\cite{s2r2}. Also, there is a description of the Maxwell sets corresponding to the symmetries $\eps^i$, $i = 1, 2, 3$: 
\begin{multline*}
\MAX^i = \{(\lam,t) \in N \mid \exists \ \b \in S^1: (\tlam, t) 
= \eps^i \circ \Fb(\lam,t), \  
\Exp(\lam, s) \not\equiv \Exp(\tlam, s), \ \Exp(\lam, t) =  
\Exp(\tlam, t)\}.
\end{multline*}
In particular, there is proved the following theorem for the symmetry $\eps^1$.
\begin{theoremSachkov}[\cite{s2r2}]
\label{th:Max1}
Suppose $t>0$ and $Q_s = (x_s, y_s, R_s) = \Exp(\lam,s)$ is an extremal trajectory such that the elastica $\{(x_s, y_s) \mid s \in [0,t]\}$ is nondegenerate, is not centered at an inflection point, and satisfies the following equation:
\begin{equation}
q_3(t) = 0. \label{eq:max1}
\end{equation}
Then $(\lam,t) \in \MAX^1$. Therefore the trajectory $Q_s$, $s \in [0, \tilde{t}]$ is not optimal for any $\tilde{t} > t$.
\end{theoremSachkov}
There is a similar theorem for the symmetry $\eps^2$.
\begin{theoremSachkov}[\cite{s2r2}]
\label{th:Max2}
Suppose $t>0$ and $Q_s = (x_s, y_s, R_s) = \Exp(\lam,s)$ is an extremal trajectory such that the elastica $\{(x_s, y_s) \mid s \in [0,t]\}$ is nondegenerate, is not centered at a vertex, and satisfies the following equation:
\begin{equation}
(x q_1 + y q_2)(t) = 0. \label{eq:max2}
\end{equation}
Then $(\lam,t) \in \MAX^2$. Therefore the trajectory $Q_s$, $s \in [0, \tilde{t}]$ is not optimal for any $\tilde{t} > t$.
\end{theoremSachkov}
In Sections~\ref{sec:g1}--\ref{sec:asymp_max} we study asymptotics of the first Maxwell times $t_1$ and $t_2$ corresponding to the Maxwell points $(\lam, t) \in \MAX^1$ and $(\lam, t) \in \MAX^2$, near stable equilibrium of mathematical pendulum~\eq{pend}. Note that there is a similar theorem for the the symmetry $\eps^3$ in~\cite{s2r2}, but we do not consider the Maxwell set $\MAX^3$ since it has lesser dimension than $\MAX^1$ and $\MAX^2$. 
\section{Asymptotics of extremal trajectories and \\ limit behavior of Maxwell sets $\MAX^1$ and $\MAX^2$}
We study optimality of extremal trajectories in the plate-ball problem. Due to complexity of parametric equations of these
trajectories it is very difficult to study this problem in the general case. Therefore we study the asymptotic case, corresponding to small oscillations of pendulum~(\ref{pend}). In this case, the sphere rolls along the curves close to small-amplitude sinusoids. 

We consider the Hamiltonian system of the Maximum Principle near the stable equilibrium $\theta=c=0$ of mathematical pendulum~(\ref{pend}). The asymptotics as $\rho_0^2=\theta_0^2 + c_0^2 \to 0$ of the solutions $x(t)$, $y(t)$, $q_0(t)$, $q_1(t)$, $q_2(t)$, $q_3(t)$ of this system is presented in~\cite{masht1}. It is well known that mathematical pendulum~(\ref{pend}) is a harmonic oscillator in the asymptotic case. The Hamiltonian system is simplified by change of variables corresponding to the symmetry of the system ''rotation by angle $\alpha$'', which is defined as follows:
\begin{eqnarray*}
&(t, \theta,c,\alpha,r,x,y,u_1,u_2,q_0,q_1,q_2,q_3)\to(s,\theta,d,\alpha,m,\bx,\by,\bu_1,\bu_2,\bq_0,\bq_1,\bq_2,\bq_3),\nonumber \\
&s = mt,\quad d = c/m, \quad m = \sqrt{r},\nonumber \\
&\vectortwo{u_1}{u_2} = A(\alpha) \vectortwo{\bu_1}{\bu_2},\text{ where $A(\alpha) = \matrixrotation{\alpha}$.} \nonumber
\end{eqnarray*}
\begin{eqnarray}
&\vectortwo{x}{y} = A(\alpha) \vectortwo{\bx}{\by},\quad
\vectortwo{q_1}{q_2} = A(\alpha) \vectortwo{\bq_1}{\bq_2},\quad
\begin{cases}
q_0 &= \bq_0,\\
q_3 &= \bq_3. \label{xyqtilde}
\end{cases}
\end{eqnarray} 
Asymptotics for the elastica $\bx, \by$ and the components $\bq_i$ of the quaternion $\bq$ are the following: 
\begin{align}
\bx(s) = & \frac{s}{m} + O(\rho_0^2), \qquad
\by(s) = \frac1m (\theta_0 \sin s  + d_0 (1-\cos s)) + 
O(\rho_0^2) \label{xytilde}\\
\bq_0(s) = &\cos \frac{s}{2 m} + O(\rho_0^2),\qquad 
\bq_2(s) = -\sin \frac{s}{2 m} + O(\rho_0^2), \label{q0q2tilde}\\ 
\bq_1(s) = & \frac{1}{2(m^2-1)}(m \cos \frac{s}{2 m} \sin s - (1+\cos s) \sin\frac{s}{2 m}) \theta_0 + \nonumber\\
               &+ \frac{1}{2(m^2-1)}(m(1-\cos s)\cos\frac{s}{2 m} - \sin s \sin\frac{s}{2 m}) d_0 + O(\rho_0^2), \label{q1tilde}\\
\bq_3(s) =& \frac{1}{2(m^2-1)}((-1+\cos s)\cos\frac{s}{2 m} + m \sin s \sin\frac{s}{2 m}) \theta_0  + \nonumber\\
               &+ \frac{1}{2(m^2-1)}(\sin s \cos\frac{s}{2 m} - m (1+\cos s)\sin\frac{s}{2 m}) d_0 + O(\rho_0^2). \label{q3tilde}
\end{align}
These expansions have a removable singularity at $m = 1$. Up to $O(\rho_0^2)$ the curve $(x,y)$ is a sinusoid of small amplitude $\frac{\rho_0}{m}$. Complexity of the formulas for the asymptotics of the quaternion is related to different frequencies of trigonometric functions. This implies existence of ''resonance'' instants. At such instants plots of the asymptotics of the first Maxwell times $t_1$ and $t_2$, corresponding to the symmetries $\eps_1$ and $\eps_2$ respectively, have vertical tangent lines, see below. In the following sections we use \eq{xytilde}--\eq{q3tilde} to study asymptotics of Maxwell points as $\rho_0 \to 0$.

To examine optimality of extremal trajectories in asymptotic case we study the first Maxwell times $t_1$ and $t_2$. They are first values of time, when extremal trajectory arrives at the Maxwell sets $\MAX^1$ and $\MAX^2$ respectively.
  
By Theorem~\ref{th:Max1} the equation $q_3=0$ implies existence of a Maxwell point for the elasticae that are nondegenerate and not centered at an inflection point. In view of (\ref{xyqtilde})--(\ref{q3tilde}), the leading term of the function $q_3(t)$ has a root if there holds the following equality: 
\begin{equation*}
\frac{d_0 \cos p - \theta_0 \sin p}{m^2-
1}(\cos\frac{p}{m} \sin p - m \cos p \sin \frac{p}{m}) =0, \quad p =\frac{s}{2}.
\label{q3}
\end{equation*}
Roots of the factor $d_0 \cos p - \theta_0 \sin p$ has a simple geometric meaning for the sinusoid $(\bx^0(s), \by^0(s)) = (s/m,(\theta_0 \sin s + d_0(1 - \cos s))/m)$ (the leading term of the asymptotics for the elastica $(\bx(s), \by(s))$), and therefore for the sinusoid $(x^0(s),y^0(s)) = (\cos \a \, \bx^0(s) + \sin \a \, \by^0(s), -\sin \a \, \bx^0(s) + \cos \a \, \by^0(s))$ (the leading term of the asymptotics for the elastica $(x(s),y(s))$ as $\rho_0 \to 0$). It can easily be checked that  the sinusoid $\{(x^0(\s),y^0(\s)) \mid \s \in [0, s]\}$ is centered at the inflection point if and only if $d_0 \cos p - \theta_0 \sin p = 0$.

Thus we study roots of the factor $(\cos\frac{p}{m} \sin p - m \cos p \sin \frac{p}{m})/(m^2-1)$. This factor has an unremovable singularity if $m = 0$, and it does not vanish for any $p \neq 0$ if $m = 1$. Therefore we consider the function  
\begin{equation}
g_1(p,m) = \cos\frac{p}{m} \sin p - m \cos p \sin \frac{p}{m}, \quad m 
\in (0,1) \cup (1,+\infty), \quad p > 0,
\label{g1}
\end{equation}
and study its minimal positive root
\begin{equation}
p_1(m) = \min\{p>0|g_1(p,m)=0\}.
\label{fp1}
\end{equation}
In Section \ref{sec:g1} we study the function $p_1(m)$. We obtain two-sided estimate of $p_1(m)$ and prove that $p_1(m)$ is monotone and differentiable for $m>0$, $m\neq 1$.

Similarly, we examine the first Maxwell time $t_2$ (the first instant of time when an extremal trajectory rich the Maxwell set $\MAX^2$). By Theorem~\ref{th:Max2} the equation $x q_1 + y q_2=0$ implies existence of a Maxwell point for the elasticae that are nondegenerate and not centered at the top. In view of (\ref{xyqtilde})--(\ref{q1tilde}), the function $x q_1 + y q_2$ is equal to zero up to $O(\rho_0^2)$ if there holds the following equality    
\begin{equation*}
\frac{d_0 \sin p + \theta_0 \cos p}{m (m^2-1)}g_2(p,m)=0,\quad p =\frac{s}{2},
\label{eq:asympmax2}
\end{equation*}
where
\begin{equation}
\label{eq:g2}
g_2(p,m)=m p \cos \left(\frac{p}{m}\right)\sin p - (p\cos p + (m^2 -1)\sin p)\sin(\frac{p}{m}).
\end{equation}
If an elastica is not centered in the vertex, then the factor $d_0 \sin p + \theta_0 \cos p$ is bounded away from zero. Therefore, we are interested in a minimal positive root of the function $g_2(p,m)=0$:
\begin{equation}
\label{eq:p2}
p_2(m)= \min\{p>0|g_2(p,m)=0\}.
\end{equation}
In Section \ref{sec:g2} we study the function $p_2(m)$. We obtain two-sided estimate of $p_2(m)$ and prove that $p_2(m)$ is a continuous function for $m>0$, $m\neq 1$.
\section{Study of the function $p_1(m)$}
\label{sec:g1}
In this section we prove a two-sided estimate of a minimal positive root of the equation $g_1(p, m)=0$ (see (\ref{g1})).
Our aim is to find or estimate as accurately as possible for any $m\in(0,1)\cup(1,+\infty)$ the value of $p_1(m)$ defined in (\ref{fp1}). The main results on this problem are summarized in Theorem \ref{th:p1}.
\begin{theorem}\label{th:p1}
For any $m \in (0,1) \cup (1,+\infty)$ the function $g_1(p,m)$ has a
minimal positive root $p_1(m)$ satisfying the following properties:
\begin{enumerate}
\item[a)] The function $p_1(m)$ is continuous and increasing for $m\in(0,1)$, coincides with $\frac{\pi m}{1-m}$ at the points $\{m = \frac{k}{k+2}|k \in \N\}$; is continuous and decreasing for $m\in(1,+\infty)$, coincides with $\frac{\pi m}{m-1}$ at the points $\{m = \frac{k+2}{k}|k \in \N\}$.
\item[b)] The function $p_1(m)$ is continuously differentiable at all points of the interval $(0,1)$, except for the set $\{m=\frac{k}{k+1}|k \in \N\}$, where its derivative is equal to $+\infty$; is continuously differentiable at all points of the interval $(1,+\infty)$, except for the set $\{m=\frac{k+1}{k}|k \in \N\}$,  where its derivative is equal to $-\infty$.
\item[c)] For $m \in (\frac13,1)$ the plot of $p = p_1(m)$ ''wraps'' the hyperbole $p = \frac{\pi m}{1-m}$ as follows:
\begin{equation}\label{p1hypm01}
\begin{cases}
p_1(m)<\frac{\pi m}{1-m} \text{ for } m \in \left(\frac{2 k -1}{2 k +1},\frac{k}{k +1} \right),\\
\frac{\pi m}{1-m} < p_1(m) \text{ for } m \in \left(\frac{k}{k +1},\frac{2 k + 1}{2 k +3} \right), \quad k \in \N.
\end{cases}
\end{equation}
For $m \in (1,3)$ the plot of $p = p_1(m)$ ''wraps'' the hyperbole $p = \frac{\pi m}{m-1}$ as follows:
\begin{equation}\label{p1hypmgr1}
\begin{cases}
\frac{\pi m}{m-1} < p_1(m) \text{ for } m \in \left(\frac{2 k +3}{2 k +1},\frac{k+1}{k} \right),\\
p_1(m)<\frac{\pi m}{m-1} \text{ for } m \in \left(\frac{k+1}{k},\frac{2k+1}{2 k -1} \right) , \quad k \in \N.
\end{cases}
\end{equation}
\item[d)]Estimates (\ref{p1hypm01}),(\ref{p1hypmgr1}) are supplemented by estimates from the other side. For $\forall k \in \N$ define 
\begin{eqnarray*}
& h_1(m) = \min\left(\frac{1+m}{2(1-m)}-k, \sqrt[3]{2 k + 1- \frac{1+m}{1-m}},\frac{1+m}{\pi (1-m)} \arcsin\frac{1-m}{1+m} \right), m \in \left(\frac{2 k -1}{2 k +1},\frac{k}{k+1} \right),	\\
& h_2(m) = \min\left(\sqrt[3]{\frac{1+m}{1-m}- (2 k+1)}, k+1 - \frac{1+m}{2(1-m)},\frac{1+m}{\pi (1-m)} \arcsin\frac{1-m}{1+m} \right), m \in \left(\frac{k}{k +1},\frac{2 k+1}{2 k+3} \right).
\end{eqnarray*}
The following inequalities are valid for $m \in \left(\frac13,1\right)$:
\begin{equation}\label{p1hyp2m01}
\begin{cases}
\frac{\pi m}{1-m} - \frac{\pi m}{1+m} h_1(m) < p_1(m) \text{ for } m \in \left(\frac{2 k -1}{2 k +1},\frac{k}{k +1} \right),\\
p_1(m)<\frac{\pi m}{1-m}+\frac{\pi m}{1+m} h_2(m) \text{ for } m \in \left(\frac{k}{k +1},\frac{2 k + 1}{2 k +3} \right), \quad k \in \N.
\end{cases}
\end{equation}
The following inequalities are valid for $1< m < 3$:
\begin{equation}\label{p1hyp2mgr1}
\begin{cases}
p_1(m)<\frac{\pi m}{m-1}+\frac{\pi m}{1+m} h_2(\frac{1}{m}) \text{ for } m \in \left(\frac{2 k +3}{2 k +1},\frac{k+1}{k} \right),\\
\frac{\pi m}{m-1} - \frac{\pi m}{1+m} h_1(\frac{1}{m}) < p_1(m) \text{ for } m \in \left(\frac{k+1}{k},\frac{2k+1}{2 k -1} \right), \quad k \in \N.
\end{cases}
\end{equation}
\item[e)]For $0<m<\frac13$ the function $p_1(m)$ has the following estimates:
\begin{equation}\label{p1m013}
\max\{\rho m, \frac{\pi m}{1-m}\}<p_1(m)<\frac{3 \pi m}{2},
\end{equation}
where $\rho = 4.493409\dots$ is the root of the equation $\tan x = x$ that lies in the range $\pi<x<\frac{3 \pi}{2}$.

For $m>3$ the function $p_1(m)$ has the following estimates:
\begin{equation}\label{p1gr3}
\max\{\rho, \frac{\pi m}{m-1}\}<p_1(m)<\frac{3 \pi}{2}.
\end{equation}
\item[f)] At the point $p = p_1(m)$ the function $p \mapsto g_1(p,m)$ changes its sign.
\end{enumerate} 
\end{theorem}
\begin{figure}
\includegraphics[width=\textwidth]{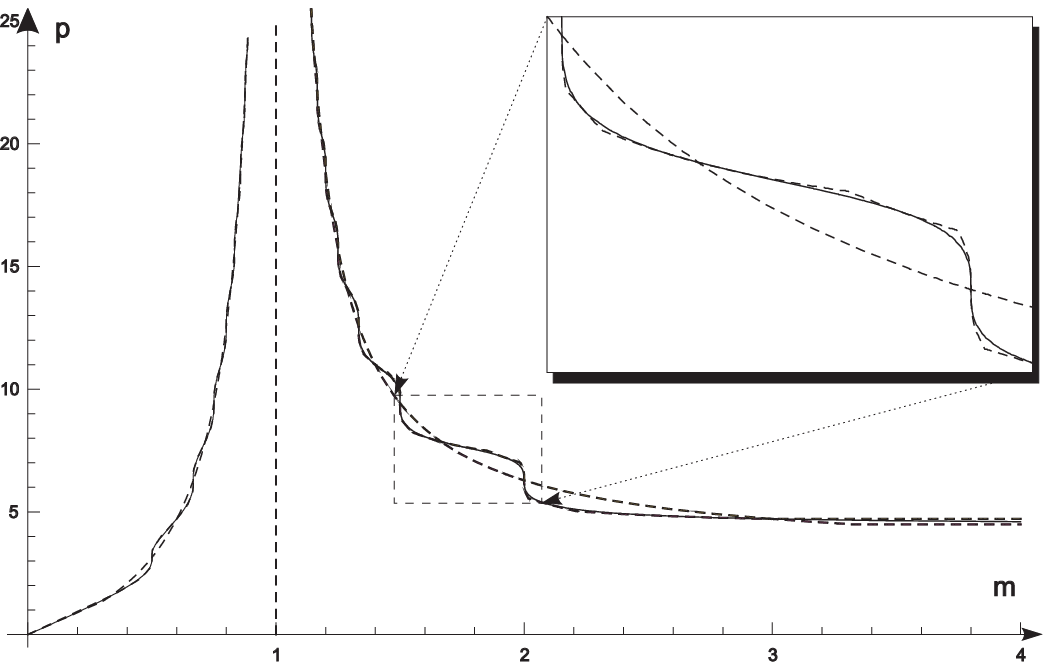}
\caption{Plot of the function $p_1(m)$ with two-sided estimates}
\label{fig:p1bounds}
\end{figure}
Figure \ref{fig:p1bounds} shows the plot of $p_1(m)$ with two-sided estimates (dashed lines).
\begin{remark}
If $m \in [\frac{2 k -1}{2 k +1},\frac{k}{k +1}]$, then the values of $s(m) = \frac{1+m}{1-m}$ range over the interval $[2k,2k+1]$, $s\left(\frac{2 k -1}{2 k +1}\right) = 2 k$, $s\left(\frac{k}{k +1}\right) = 2 k+1$. Consequently, 
$$0\leq h_1(m)<0.3264, \lim_{m\to\frac{2 k -1}{2 k +1}+0}h_1(m)=0, \lim_{m\to\frac{k}{k +1}-0}h_1(m)=0.$$
If $m \in [\frac{k}{k +1},\frac{2 k+1}{2 k +3}]$, then the values of $s(m)$ range over the interval $[2k+1,2k+2]$, $s\left(\frac{2 k +1}{2 k +3}\right) = 2 k+2$. Consequently, 
$$0\leq h_2(m)<0.3244, \lim_{m\to\frac{k}{k +1}+0}h_2(m)=0, \lim_{m\to\frac{2 k+1}{2 k +3}-0}h_2(m)=0.$$  
\end{remark}
\begin{remark}
From (\ref{p1hypm01})--(\ref{p1hyp2mgr1}) it follows that $p_1(m)$ tends to $+\infty$ as $m\to 1 \pm 0$.
\end{remark}
\subsection{Estimation of $p_1(m)$ for $0<m<\frac{1}{3}$}
Let us prove that equation (\ref{g1}) has a root in the interval $0<p<\frac{3 \pi m}{2}<\frac{\pi}{2}$. In this interval the function $\sin p$ does not vanish. The function $\sin \left(\frac{p}{m}\right)$ vanishes only at $p=\pi m$ but this point is not a root of equation (\ref{g1}), since $g_1(\pi m,m)=-\sin(\pi m) \neq 0$. Therefore, dividing the equation $g_1(p,m)=0$ by their product we obtain in the interval $0<p<\frac{\pi}{2}$ equivalent equation
\begin{equation}
\label{eq:g1symv}
\cot \left(\frac{p}{m}\right)= m\cot p \Leftrightarrow f\left(\frac{p}{m}\right)=f(p) ,\text{ where } f(x)=x\cot(x).
\end{equation}
Put $y=\frac{p}{m}$. Then equation (\ref{eq:g1symv}) has the form $f(y)=f(my)$. The function $f$ decreases over the intervals $(0,\pi)$, $(\pi, 2\pi)$ and satisfies the following equalities:
$$\lim_{y\to 0+} f(y) =1, \quad \lim_{y\to \pi-0} f(y) =-\infty, \quad \lim_{y\to \pi+0} f(y) =+\infty, \quad f(\frac{3\pi}{2}) =0.$$ 
It follows from decrease of $f(y)$ over the interval $0<y<\pi$ that the equation $f(y)=f(my)$ has no roots for $0<m<1$ in this interval. But a root exists in the interval $\pi<y<\frac{3\pi}{2}$ since the difference $f(y)-f(my)$ is continuous and tends to $+\infty$ as $y\to\pi+0 $ and is negative at $\frac{3 \pi}{2} $ (because of condition $m <\frac{1}{3}$ we have $f \left(\frac{3 \pi m }{2} \right)> 0 $). Denote this root by $y_0$. Since $f(my)<1$ for $0 <y <\frac{\pi}{m}$, it is clear that $y_0(m)$ exceeds the number $\rho$ defined by the equation $f(\rho) = 1$, $\pi <\rho <\frac{3\pi}{2}$. Consequently, $\rho<y_0(m)<\frac {3 \pi}{2}$. Thus for $0<m<\frac{1}{3}$ we obtain the inequality:
\begin{equation}
\label{eq:p1m013}
\rho m < p_1(m) < \frac{3\pi m}{2},
\end{equation}
where $\rho=4.493409\dots$ is the root of the equation $\tan x = x$ that lies in the interval $\pi<x<\frac{3\pi}{2}$.

From (\ref{eq:p1m013}) it follows immediately that for $0<m<\frac13$ the root $p_1(m)$ lies in the interval $0<p<\frac{\pi}{2}$, as it was claimed at the beginning of this section. The lower estimate in (\ref{eq:p1m013}) needs to be improved as $m\rightarrow \frac{1}{3}-0$. This improvement will be done below in Section \ref{sec_4_3}.
\subsection {Reducing the problem to finding (estimation) of the minimal positive root of a simpler equation $\tilde{g}(x,s)=0$, $s>1$, $\tilde{g}(x,s)=s \sin x - \sin (sx)$.} \label{subsectildeg1}
It can easily be checked that
\begin{eqnarray*}
& g_1(p,m)=\frac{m+1}{2}\left(\sin p \cos \frac{p}{m}-\cos p \sin \frac{p}{m}\right)-\frac{m-1}{2}\left(\sin p \cos \frac{p}{m}+ \cos p \sin \frac{p}{m}\right).
\end{eqnarray*}
Since $\sin p \cos \frac{p}{m} \pm \cos p \sin \frac{p}{m} = \sin \left(p\pm \frac{p}{m}\right)$, we see that the function $g_1$ admits the following representation: 
\begin{eqnarray*}
& g_1(p,m)=\frac{m+1}{2}\sin  \left( p \left (1-\frac{1}{m}\right) \right)- \frac{m-1}{2} \sin  \left(p\left(1+\frac{1}{m}\right)\right)= \frac{1-m}{2} \sin  \left(p\frac{m+1}{m}\right)-\frac{1+m}{2} \sin  \left(p\frac{1-m}{m}\right).
\end{eqnarray*}
Let us introduce a parameter $s=\frac{1+m}{1-m}$ and a new variable $x=p\frac{1-m}{m}$. Now express $g_1(p,m)$:
\begin{eqnarray*}
g_1(p, m) = \frac{1-m}{2} \sin (sx) - \frac{1+m}{2} \sin x = \frac{1-m}{2}\left(\sin (sx) - s \sin x \right)= \frac{m-1}{2}\tilde{g}(x,s).
\end{eqnarray*}
Thus, we study the function $$x_1(s)= \min \{x>0| \tilde{g}(x,s)=0\} \text{ for } s>1.$$
\subsection {Absence of roots of $\tilde{g}(x,s)$ in the half-interval $0<x \leq \pi$ for $s \in (1,2)$. Improvement of estimate (\ref{eq:p1m013})}\label{sec_4_3}
Since $\tilde{g}(0,s)=0$ and $\tilde{g}(\pi, s)=-\sin (\pi s)> 0$, we see that from existence of a root in the interval $(0, \pi)$  (at $\pi$ there is no root) it follows that $\tilde{g}(x,s)$ in $[0, \pi]$ has a minimum at some point in the interval $(0, \pi)$. Derivative of $\tilde{g}$ at this point is equal to zero and the value of $\tilde{g}$ is nonpositive. Thus, existence of a root $\tilde{g}(x,s)$ in the interval $0<x<\pi$ for $s\in(1,2)$ implies existence of a critical point in $(0,\pi)$, where the value of the function $\tilde{g}$ is nonpositive. 

Let us find critical points (roots of the derivative by $x$) of the function $\tilde{g}(x,s)$ for any $s$. We have $\tilde{g}'_x(x,s)=s(\cos x - \cos (sx))$. Consequently,
\begin{align}
\tilde{g}'_x(x,s)=0 \Leftrightarrow \cos x = \cos sx \Leftrightarrow 
\begin{cases}
x = s x - 2\pi n_1,\quad n_1 \in \Z\\
-x = s x - 2\pi n_2,\quad n_2 \in \Z
\end{cases}
\Leftrightarrow \nonumber  \\
\Leftrightarrow x \in\{\frac{2 \pi n_1}{s-1}|n_1 \in \Z\},\qquad x \in\{\frac{2 \pi n_2}{s+1}|n_2 \in \Z\}. \label{eq:criticalPoints_g1}
\end{align}

Since $s\in(1, 2)$, we see that the interval $(0, \pi)$ contains only one point of form  (\ref{eq:criticalPoints_g1}). Namely,  $x=\frac{2\pi}{s+1}$. Compute the value of $\tilde{g}$ at this point: 
\begin{eqnarray*}
\tilde{g}\left(\frac{2\pi}{s+1},s\right)= 
&& s \sin \left(\frac{2\pi}{s+1}\right)- \sin \left(\frac{2\pi s}{s+1}\right)= s \sin \left(\frac{2\pi}{s+1}\right)+\\
&&+\sin \left(2\pi -\frac{2\pi s}{s+1} \right)=s \sin \left(\frac{2\pi}{s+1}\right)+ \sin\left(\frac{2\pi}{s+1}\right)=\\
&&=(s+1)\sin\left(\frac{2\pi}{s+1}\right)>0.
\end{eqnarray*}
This inequality contradicts the conclusion (presented above) from the assumption of existence of a root $\tilde{g}$ in $(0, \pi)$. Hence, $\tilde{g}(x, s)$ has no root in $(0, \pi]$ for $s\in(1, 2)$ and we obtain the following inequality:
$$ x_1 (s)> \pi \text{ for } s\in(1, 2) \Rightarrow p_1 > \frac{\pi m}{1-m},\quad m\in(0, \frac{1}{3}).$$
So, we improve estimates (\ref{eq:p1m013}) as follows:
\begin{equation}
\max\left(\rho m, \frac{\pi m}{1-m}\right)< p_1(m)<\frac{3\pi m }{2},\quad  0<m<\frac{1}{3}.
\label{eq:p1m013prec}
\end{equation}
It is important that the functions in lower and upper bound (\ref{eq:p1m013prec}) have the same limits as $m\rightarrow 0+$ and as $m\rightarrow\frac{1}{3}-0$.
Also it is easy to study the case $m=\frac{1}{3}$ (i.e. $s=2$):
$$p_1\left(\frac{1}{3}\right)=\frac{\pi}{2}.$$
\subsection {Preliminary two-sided estimate of $x_1(s)$, $s>2$}
First we prove that
\begin{equation}
\pi - \arcsin \left(\frac{1}{s}\right) \leq x_1(s) \leq \pi + \arcsin \left(\frac{1}{s}\right) \quad  \forall s >2.
\label{eq:p1_general}
\end{equation}
Existence of a root on the interval $\pi-\arcsin \left(\frac{1}{s}\right)\leq x \leq \pi+ \arcsin \left(\frac{1}{s}\right)$ immediately follows from the relations
\begin {align}
\tilde{g} \left(\pi \pm \arcsin \left(\frac{1}{s}\right), s \right) = \pm 1-\sin\left(\pi s \mp s \arcsin\left(\frac{1}{s}\right)\right)\Rightarrow \nonumber \\
\tilde{g}\left(\pi - \arcsin \left(\frac{1}{s}\right),s\right)\geq 0, \quad  \tilde{g}\left(\pi+\arcsin \left(\frac{1}{s}\right),s\right)\leq 0.
\label{eq:g1_6}
\end {align}

Therefore $x_1(s)\leq \pi + \arcsin \left(\frac{1}{s}\right)$. To prove estimate (\ref{eq:p1_general}) it remains to check that for any $s>2$ the function $\tilde{g}(x,s)$ has no roots for $0<x< \pi- \arcsin \left(\frac{1}{s}\right)$. It is clear that $s \sin x> 1$ when $\arcsin \left(\frac{1}{s}\right)<x< \pi -\arcsin \left(\frac{1}{s}\right)$.  Hence, $\tilde{g}(x,s)>0$. In the half-interval  $0<x\leq \arcsin\left(\frac{1}{s}\right)$ the function $\tilde{g}(x,s)$ increases (due to decrease of $\cos t$ on $0\leq t\leq \pi$ we have $\cos (s x)< \cos x$ for $0< x< \arcsin \left(\frac{1}{s}\right)\Rightarrow \tilde{g}'_x(x, s)=s (\cos x - \cos (s x))>0$). Since $\tilde{g}(0, s)=0$, we see that the function $\tilde{g}(x, s)$ is positive for $0< x \leq \arcsin \left(\frac{1}{s}\right)$. So, two-sided inequality (\ref{eq:p1_general}) is proved. Now we improve it.
\begin{lemma}\label{lemma1}
If $s\in(2k, 2k+1)$, $k\in \N$, then the function $\tilde{g}(x, s)$ has a unique root in the half-interval  $\pi-\arcsin\left(\frac{1}{s}\right)\leq x <\pi$. If $s\in(2k+1, 2k+2)$, $k\in \N$, then the function $\tilde{g}(x,s)$ has no roots on the interval $\pi-\arcsin \left(\frac{1}{s}\right)\leq x \leq \pi$ and has a unique root in the half-interval $\pi< x \leq \pi + \arcsin\left(\frac{1}{s}\right)$. 
\end {lemma}
\begin{corollary} The function $x_1(s)$ has the following estimates
\begin {align}
\begin{cases}
\pi - \arcsin \left(\frac{1}{s}\right)\leq x_1(s)< \pi, \quad s\in(2k, 2k+1), \\
\pi< x_1(s)\leq \pi+\arcsin \left(\frac{1}{s}\right),  \quad s\in(2k+1, 2k+2),\quad  k\in \N.
\end{cases}
\label{corol1}
\end {align}
\end{corollary}
\begin{proof}[Proof of Lemma \ref{lemma1}.] Assume that $s\in(2k, 2k+1)$, where $k \in \N$. Then, since $\tilde{g}(\pi, s) = - \sin (\pi s)<0$ and $\tilde{g}\left(\pi - \arcsin \left(\frac{1}{s}\right),s\right)\geq 0$, it follows that there exists a root of $\tilde{g}(x, s)$ in the half-interval $\pi - \arcsin \left(\frac{1}{s}\right)\leq x< \pi$. Uniqueness of the root follows from the constant sign of the derivative $\tilde{g}'_x(x,s)$ that is equivalent to absence of its roots. Let us show that for all $s>2$ and $x \in I_s=\left[ \pi-\frac{\pi}{2s}, \pi \right]$ we have $\tilde{g}'_x(x,s)\neq0$. Notice that $I_s$ contains the interval $[\pi - \arcsin \left(\frac{1}{s}\right), \pi]$, since for any $s>1$ we have $s\arcsin \left(\frac{1}{s}\right)< \frac{\pi}{2}$. For all $2k <s< 2k+1$ we see that $I_s$ does not contain any critical point (root of $\tilde{g}'_x$) presented in (\ref{eq:criticalPoints_g1}), since 
\begin{align*}
\frac{2\pi(k-1)}{s-1}<\pi-\frac{\pi}{2 s}, \quad \frac{2\pi k}{s-1}>\pi, \quad \frac{2\pi k}{s+1}<\pi-\frac{\pi}{2s}, \quad \frac{2\pi(k+1)}{s+1}>\pi.
\end {align*}
The second and the fourth inequalities are absolutely obvious. Let us check the first and the third inequalities. We have
\begin{align*}
\frac{2\pi(k-1)}{s-1}< \pi-\frac{\pi}{2s}\quad \Leftrightarrow \quad \frac{2k-2}{s-1}< 1-\frac{1}{2s}\Leftrightarrow \quad 2k-2< s-1-\frac{s-1}{2s} \quad \Leftrightarrow \quad 2k<s+\frac{1}{2}+\frac{1}{2s},
\end {align*}
that is true, since even $s>2k$. Further
\begin{align*}
\frac{2\pi k}{s+1}< \pi-\frac{\pi}{2s} \quad \Leftrightarrow \quad \frac{2k}{s+1}<1-\frac{1}{2s}\quad \Leftrightarrow \quad 2k<s+1-\frac{s+1}{2s} \quad \Leftrightarrow \quad 2k< s+\frac{1}{2}-\frac{1}{2s}.
\end {align*}
The last inequality is also true, since $s> 2k$ and $\frac{1}{2}-\frac{1}{2s}>0$.

Now, assume that $s\in(2k+1, 2k+2)$, where $k \in \N$. Let us prove that $\tilde{g}(x, s)$ has no roots on $I_s$. We have 
\begin{eqnarray*}
\tilde{g}\left(\pi-\frac{\pi}{2s}, s \right)= && s \sin\left(\frac{\pi}{2s}\right)-\sin \left(\pi s - \frac{\pi}{2} \right)= 
                                              s\sin \left(\frac{\pi}{2s}\right)+ \cos (\pi s )> 1+\cos(\pi s)>0,\\
\tilde{g}(\pi, s)= && -\sin(\pi s)>0.
\end{eqnarray*}
Thus, existence of a root on $I_s$ would imply existence of a critical point $\tilde{x} \in I_s$ such that $\tilde{g}(\tilde{x},s)\leq 0$. It is easy to prove that $I_s$ contains only two critical points (\ref{eq:criticalPoints_g1}), namely   
$$\frac{2\pi k}{s-1}, \qquad \frac{2\pi(k+1)}{s+1}.$$
Let us calculate the values $\tilde{g}\left(\frac{2\pi k }{s-1}, s\right)$, $\tilde{g}\left(\frac{2\pi (k+1)}{s+1}, s\right)$. We have 
\begin{multline*}
\tilde{g}\left(\frac{2\pi k }{s-1}, s\right)=
s \sin \left(\frac{2\pi k }{s-1}\right)- \sin \left(\frac{2 \pi k s}{s-1}\right)= \\
=s \sin \left(\frac{2 \pi k}{s-1}\right)- \sin \left(2\pi k +\frac{2 \pi k }{s-1}\right)=(s-1)\sin\left(\frac{2\pi k}{s-1}\right)>0,
\end{multline*}
\begin{multline*}
\tilde{g}\left(\frac{2 \pi (k+1)}{s+1}, s\right)= s \sin \left(\frac{2\pi(k+1)}{s+1}\right)-\sin\left(\frac{2\pi(k+1)s}{s+1}\right)=\\
=s\sin\left(\frac{2\pi(k+1)}{s+1}\right)- \sin \left(2\pi(k+1)- \frac{2\pi(k+1)}{s+1}\right)=(s+1)\sin\left(\frac{2\pi(k+1)}{s+1}\right)>0.
\end{multline*}
Thus we see that $\tilde{g}(x,s)$ has no roots in $I_s$. Then, since $\tilde{g}(\pi,s)=-\sin(\pi s)>0$ and $\tilde{g}(\pi +\arcsin\frac{1}{s})\leq 0$, it follows that there exists a root of $\tilde{g}(x, s)$ in the half-interval $\pi< x \leq \pi + \arcsin\left(\frac{1}{s}\right)$. Uniqueness of the root follows from the inequality $\tilde{g}'_x(x,s)\neq 0$ for all $s>2$ and $x \in \hat{I}_s=\left[ \pi, \pi + \frac{\pi}{2s} \right]$. Note that $\hat{I}_s$ contains the interval $\pi \leq x_1(s)\leq \pi + \arcsin \left(\frac{1}{s}\right)$. For all $2k+1 <s< 2k+2$ we see that $\hat{I}_s$ does not contain any critical point presented in (\ref{eq:criticalPoints_g1}), since 
\begin{align*}
\frac{2\pi k}{s-1}<\pi, \quad \frac{2\pi (k+1)}{s-1}>\pi + \frac{\pi}{2 s}, \quad \frac{2\pi (k+1)}{s+1}<\pi, \quad \frac{2\pi(k+2)}{s+1}>\pi + \frac{\pi}{2 s}.
\end {align*}
The first and the third inequalities are absolutely obvious. Let us check the second and the fourth inequalities. We have
\begin{align*}
\frac{2\pi(k+1)}{s-1}> \pi+\frac{\pi}{2s}\quad \Leftrightarrow \quad \frac{2k+2}{s-1}> 1+\frac{1}{2s}
\Leftrightarrow \quad 2k+2> s-1+\frac{s-1}{2s} \quad \Leftrightarrow \quad 2k+2>s-\frac{s+1}{2s},
\end {align*}
that is true, since $\frac{s+1}{2s}>0$ and even $s<2k+2$. Further
\begin{align*}
\frac{2\pi (k+2)}{s+1}> \pi+\frac{\pi}{2s} \ \Leftrightarrow \ 2k+4>s+1+\frac{s+1}{2s}\  \Leftrightarrow \ 2k+2>s-1+\frac{s+1}{2s} \ \Leftrightarrow \ 2k+2>s+\frac{1-s}{2s}.
\end {align*}
The last inequality is also true, since $s< 2k+2$ and $\frac{1-s}{2s}<0$.
This completes the proof of Lemma~\ref{lemma1}.
\end{proof}
\begin{proposition}\label{propos:g1_1}
There holds the equality 
\begin{equation}
x_1(s)=\pi, \quad \forall s \in \N, \quad s\geq 2. \label{eq:x1pi}
\end{equation}
\end{proposition}
\begin{proof}
Let us prove (\ref{eq:x1pi}). Since $s \in \N$ it can easily be checked that $\pi$ is a root of $\tilde{g}(x,s)$. In view of (\ref{eq:p1_general}), it remains to prove that $\tilde{g}(x,s)$ has no roots in the half-interval $\pi-\arcsin\left(\frac{1}{s}\right)\leq x < \pi$. Put $y=\pi-x$. Thus we have  
\begin{eqnarray*}
& \tilde{g}(x,s) =\tilde{g}(\pi-y,s)=s\sin(\pi-y)-\sin(\pi s - y s)= s\sin y -\\
&-(-1)^s \sin (-y s)= s\sin y + (-1)^s \sin ys \geq s \sin y - \sin ys = \tilde{g}(y,s).
\end{eqnarray*}
It was proved above that $\tilde{g}(y,s)>0$ for $y\in \left( 0, \arcsin \frac{1}{s}\right]$. So, $\tilde{g}(x,s)>0$ for $\pi-\arcsin\left(\frac{1}{s}\right)\leq x < \pi$. This completes the proof of Proposition \ref{propos:g1_1}. 
\end{proof}
\begin{remark}
Note that inequalities (\ref{corol1}) turn into equalities in an infinite set of values of the parameter $s$. In fact for any $k\in\N$ there exist $s_{1,k}\in\left(2k+\frac{1}{2}, 2k+1\right)$ and $s_{2,k}\in \left(2k+1, 2k+\frac{3}{2}\right)$ such that $x_1(s_{1,k})=\pi-\arcsin\left(\frac{1}{s_{1,k}}\right)$ and $x_1(s_{2,k})=\pi+\arcsin\left(\frac{1}{s_{2,k}}\right)$. 
\end{remark}
\subsection{More accurate estimate of $x_1(s)$}
Let $\rho(z)$, $z\in \R$ be the distance from $z$ to the closest integer number. For any $s\geq 2$ define a function $a(s)$ as follows: 
\begin{eqnarray*}
a(s) = \begin{cases}
\min\left( \frac{\pi}{s} \rho \left(\frac{s}{2}\right), \arcsin\left(\frac{1}{s} \right)\right) \text{ for } 0 \leq \rho \left(\frac{s}{2}\right)<\frac{7}{16}, \\
\min\left( \frac{\pi}{s}\left( 1- 2 \rho \left(\frac{s}{2}\right)\right)^{\frac13}, \arcsin\left(\frac{1}{s} \right)\right) \text{ for } \frac{7}{16} \leq \rho \left(\frac{s}{2}\right)\leq\frac{1}{2}.
\end{cases}
\end{eqnarray*}
Note that $a(s)$ is nonnegative, is equal to zero for all integers $s$, and tends to zero as $s\to +\infty$. Absolute value of derivative $a'(s)$ is equal to $\frac{\pi}{2s}$ when $s$ is odd number and is equal to $\infty$ when $s$ is even number.   
\begin{proposition} \label{propos:g1_2}
$x_1(s)$ admits the following estimates:
\begin {align*}
\pi-a(s)\leq x_1(s)< \pi, && 2k < s < 2k+1, \\
\pi< x_1(s)\leq \pi+a(s), && 2k+1<s<2k+2, \quad k \in \N.
\end {align*}
\end{proposition}
\begin{proof}
First assume that $s\in(2k, 2k+1)$, where $k \in \N$. In view of (\ref{corol1}) and (\ref{eq:p1_general}) it remains to prove that
\begin {align} \label{eq:g1_8}
\begin{cases}
\pi-\frac{\pi}{s}\rho\left(\frac{s}{2}\right)< x_1(s), \quad s\in(2k, 2k+\frac{7}{8}), \\
\pi-\frac{\pi}{s} t^{\frac13}< x_1(s), \quad s\in (2k+\frac{7}{8}, 2k+1),
\end{cases}
\end {align}
\begin{equation} \label{eq:g1_9}
\text{ where } t=1-2\rho\left(\frac{s}{2}\right)= 1-2\left(\frac{s}{2}-k\right)= 2k+1-s.
\end{equation}
It follows from Lemma \ref{lemma1} that $\tilde{g}(x,s)$ has a unique root in the interval $0<x< \pi$. Since $\tilde{g}(\pi,s)<0$ we see that to prove (\ref{eq:g1_8}) we must prove that 
\begin {align}
\label{eq:g1_10} \tilde{g}\left(\pi - \frac{\pi}{s} \rho \left(\frac{s}{2}\right), s\right)>0, && s\in\left(2k, 2k+\frac{7}{8}\right), \\
\tilde{g}\left(\pi- \frac{\pi t^\frac{1}{3}}{s},s\right)>0, && s\in \left[2k+\frac{7}{8}, 2k+1\right).\label{eq:g1_11}
\end {align}
We have
$$\tilde{g}\left(\pi - \frac{\pi}{s}\rho\left(\frac{s}{2}\right), s \right) = s \sin \left(\pi - \frac{\pi}{s}\rho \left(\frac{s}{2}\right)\right)
- \sin \left(\pi s - \pi \rho \left(\frac{s}{2}\right)\right).$$
Note that $s=2k+ 2 \rho \left(\frac{s}{2}\right)$ for $s\in(2k, 2k+1)$. Therefore, we have 
\begin{multline*}
\tilde{g}\left(\pi - \frac{\pi}{s}\rho\left(\frac{s}{2}\right), s \right) = s \sin \left(\frac{\pi}{s}\rho \left(\frac{s}{2}\right)\right)- \sin \left(2\pi k + \pi \rho \left(\frac{s}{2}\right)\right) = s \sin \left(\frac{\pi}{s} \rho \left(\frac{s}{2}\right) \right)- \sin \left(\pi \rho \left(\frac{s}{2}\right)\right)> 0.
\end{multline*}
(We already noticed that $\sin \alpha < n \sin \left(\frac{\alpha}{n}\right)$ for $\alpha\in(0, \pi)$, $n>1$.) Thus inequality (\ref{eq:g1_10}) is proved. Let us prove (\ref{eq:g1_11}). We have
$$\tilde{g} \left( \pi - \frac{\pi t ^{\frac{1}{3}}}{s}, s \right) = s \sin \left(\pi - \frac{\pi t ^{\frac{1}{3}}}{s}\right) - \sin \left( \pi s - \pi t^{\frac{1}{3}}\right).$$
Using (\ref{eq:g1_9}), we get
\begin {multline*}
\tilde{g} \left( \pi - \frac{\pi t ^{\frac{1}{3}}}{s}, s \right) = s \sin \left(\frac{\pi t^{\frac{1}{3}}}{s}\right)- \sin \left(\pi \left(2k+1 -t \right)- \pi t^{\frac{1}{3}}\right)=\\
= s \sin \left(\frac{\pi t^{\frac{1}{3}}}{s}\right)- \sin \left(2 \pi k + \pi - \pi t - \pi t^{\frac{1}{3}}\right)= s \sin \left(\frac{\pi t^{\frac{1}{3}}}{s}\right) - \sin\left(\pi t + \pi t^{\frac{1}{3}}\right).
\end {multline*}
Therefore, we must prove the following inequality: 
\begin {equation*}
\sin \left(\pi t + \pi t^{\frac{1}{3}}\right)< s \sin \left(\frac{\pi t^{\frac{1}{3}}}{s} \right) \text{ for } 0<t\leq\frac{1}{8}, \quad s\geq 3 - \frac{1}{8}.
\end {equation*}
Note that $\pi t + \pi  t^{\frac{1}{3}}\leq \pi \left(\frac{1}{8}+ \frac{1}{2}\right)= \frac{5 \pi }{8} < 2 $ and if $0< \alpha<2$, then $\sin \alpha < \alpha-\frac{\alpha^3}{6}+\frac{\alpha^5}{120}< \alpha - \frac{\alpha^3}{6}+\frac{4\alpha^3}{120}= \alpha- \frac{2\alpha^3}{15}$.
On the other hand, $\sin \alpha > \alpha-\frac{\alpha^3}{6}$ for any $\alpha > 0$. Hence, we need prove that   
\begin{multline*}
\pi t + \pi t^{\frac{1}{3}} - \frac{2}{15}\left(\pi t + \pi t^{\frac{1}{3}} \right)^3< s \left(\frac{\pi t^{\frac{1}{3}}}{s}- \frac{\pi^3 t }{6 s^2}\right) \Leftrightarrow \pi t + \frac{\pi^3 t }{6s^2}< \frac{2}{15} \left(\pi t + \pi t^{\frac{1}{3}}\right)^3.
\end{multline*}
The right-hand side of this inequality exceeds $\frac{2}{15}\pi^3 t$. Therefore, the required inequality will be proved if we verify that
$$1+\frac{\pi^2}{6s^2}< \frac{2}{15}\pi^2 \Leftrightarrow 15< \pi^2 \left(2- \frac{2.5}{s^2}\right).$$
Since $s>2.5$, we have 
$$\pi^2 \left(2- \frac{2.5}{s^2}\right)> \pi^2 \left(2-\frac{1}{2.5}\right)= \frac{\pi^2\cdot 8}{5}>15.$$

Now consider $s\in(2k+1, 2k+2)$, where $k \in \N$. It is sufficient to prove that $\tilde{g}(\pi+a(s), s)\leq 0$. According to the definition of $a(s)$ this follows from the two inequalities 
\begin{align}
\tilde{g}\left(\pi+\frac{\pi}{s}t^{\frac{1}{3}}, s \right)<0, && s\in \left(2k+1, 2k+\frac{9}{8}\right], && t=s-(2k+1),\label{eq:g1_12} \\
\tilde{g}\left(\pi+ \frac{\pi}{s}\rho \left(\frac{s}{2}\right), s \right)<0, && s\in \left(2k+\frac{9}{8}, 2k+2 \right). \label{eq:g1_13}
\end{align}
(The inequality $\tilde{g}\left(\pi + \arcsin \left(\frac{1}{s}\right), s\right)\leq 0$ was proved above.)

In fact, inequalities (\ref{eq:g1_12}), (\ref{eq:g1_13}) are in some sense ''symmetric reflections'' of inequalities (\ref{eq:g1_10}), (\ref{eq:g1_11}) and follows from them. To verify this we prove the following implication (with account of the following lemma, the proof of Proposition \ref{propos:g1_2} is complete).
\end{proof}
\begin{lemma}\label{lemma2}
Let $\delta \in \left(0, \frac{\pi}{2}\right]$, $s_-$, $s_+ \in \R $, $1<s_-<s_+$, $s_- + s_+= 2 \nu$, where $\nu \in \N$. If $\tilde{g}\left(\pi-\frac{\delta}{s_-}, s_-\right)>0$, then $\tilde{g}\left(\pi+\frac{\delta}{s_+},s_+\right)<0.$
\end {lemma}
It is easy to show that inequalities (\ref{eq:g1_10}), (\ref{eq:g1_13}) and (\ref{eq:g1_11}), (\ref{eq:g1_12}) form exactly such pairs, as in Lemma \ref{lemma2}. According to this we have that (\ref{eq:g1_10}) implies (\ref{eq:g1_13}), and (\ref{eq:g1_11}) implies (\ref{eq:g1_12}).
\begin {proof}[Proof of Lemma \ref{lemma2}] It is sufficient to prove that the sum $\tilde{g}\left(\pi-\frac{\delta}{s_-}, s_-\right)+\tilde{g}\left(\pi+\frac{\delta}{s_+},s_+\right)$ is negative. We have  
\begin{align}
s_- \sin \left(\pi - \frac{\delta}{s_-}\right)- \sin \left(\pi s_- -\delta\right)+s_+ \sin \left(\pi+ \frac{\delta}{s_+}\right)- \sin\left(\pi s_+ + \delta\right)= \nonumber \\
= \left(s_- \sin \left(\frac{\delta}{s_-}\right)-s_+ \sin \left(\frac{\delta}{s_+}\right)\right)-\left(\sin \left(\pi s_- - \delta \right)+ \sin \left(\pi s_+ + \delta \right)\right). \label{eq:g1_14}
\end{align}
If $\delta\in\left(0,\frac{\pi}{2}\right]$, then the function $s \sin \left(\frac{\delta}{s}\right)$ increases for $1\leq s < +\infty$ and $\left(s_- \sin \left(\frac{\delta}{s_-}\right)-s_+ \sin \left(\frac{\delta}{s_+}\right)\right)<0$. Since $\sin\left(\frac{\pi\left(s_-+s_+\right)}{2}\right)= \sin \left(\pi \nu \right)$, $\nu \in \N$, we have $\left(\sin \left(\pi s_- - \delta \right)+ \sin \left(\pi s_+ + \delta \right)\right)=0$. This completes the proof of Lemma \ref{lemma2}. 
\end {proof}
Figure \ref{fig:s1bounds} shows the plot of $x_1(m)$ and obtained two-sided estimates (dashed lines).
\begin{figure}
\includegraphics[width=\textwidth]{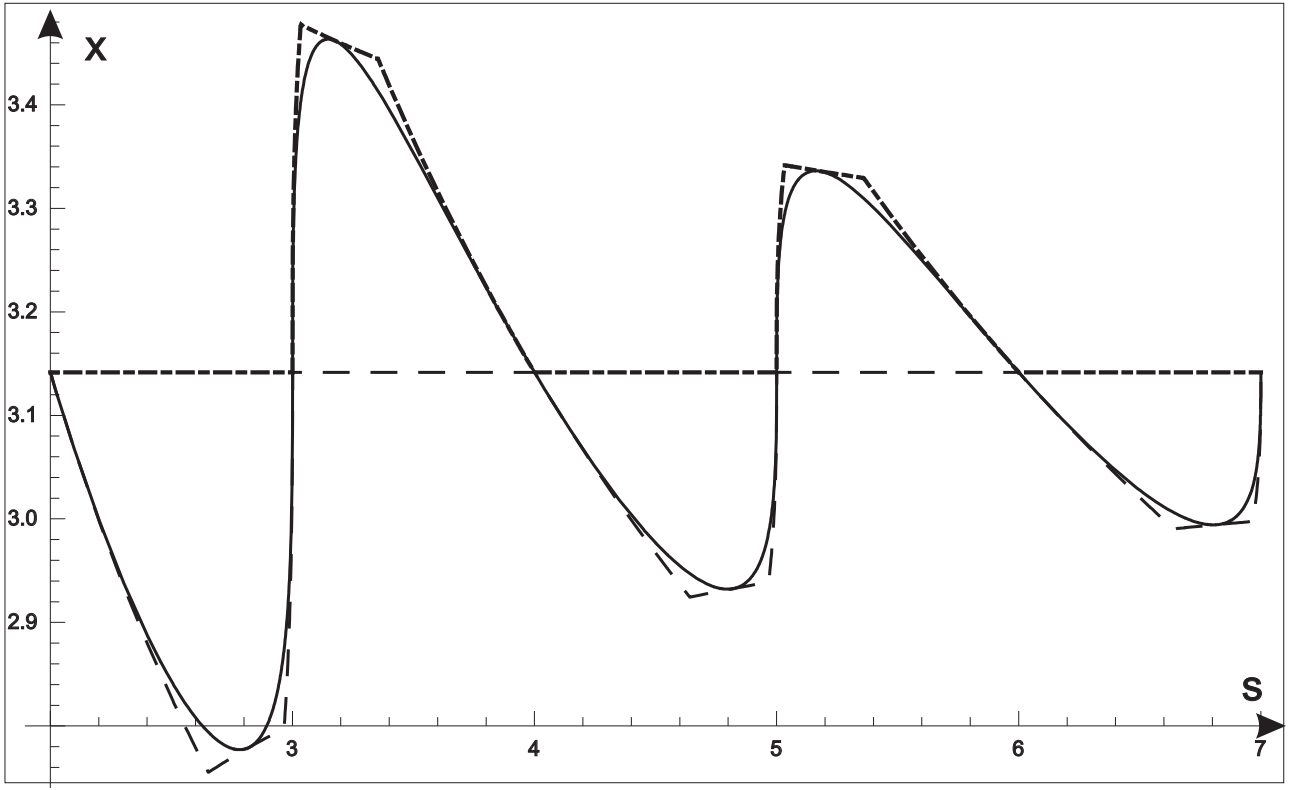}
\caption{Plot of $x_1(s)$ with two-sided estimates (dashed lines)}
\label{fig:s1bounds}
\end{figure}
Since the function $\frac{\sin t}{t}$ decreases in the interval $0<t<\pi$, we have $\arcsin\left(\frac{1}{s}\right)< \frac{\pi}{3 s}$ for $s>2$. Therefore
\begin{eqnarray*}
&& \frac{\pi\rho}{s}\geq \arcsin\left(\frac{1}{s}\right), \text{ if } \rho \geq \frac13,\\
&& \frac{\pi}{s} \sqrt[3]{1-2 \rho} \geq \arcsin\left(\frac{1}{s}\right), \text{ if } \sqrt[3]{1-2 \rho} \geq \frac13 \Leftrightarrow \rho \leq \frac{13}{27}.
\end{eqnarray*}
So if $\rho \in [\frac13, \frac{13}{27}]$, then $\arcsin\left(\frac{1}{s}\right)$ is the best estimate. Therefore in Theorem \ref{th:p1} the minimum is a well-posed operation. Returning to the original variables $p = x \frac{s-1}{2}$, $m = \frac{s-1}{s+1}$ we obtain the statement of Theorem \ref{th:p1}.
\subsection{Study of $p_1(m)$ for $m>1$}\label{subsec:g1_sym}
We use the change of variables  $\bar{p} = \frac{p}{m}$, $\bar{m} = \frac{1}{m} \in (0,1)$ to study $p_1(m)$ in the case $m > 1$. We have 
$g_1(\bar{p},\bar{m}) = -\frac{1}{m} g_1(p,m)$. Therefore, we have $g_1(\bar{p},\bar{m}) =0$ iff $g_1(p,m)=0$. 
In such a way we get the following functional equation: 
\begin{equation}
\label{eq:g1_p1sym}
p_1(m) = m p_1(\frac{1}{m}).
\end{equation}
Hence, the properties a)---f) (see Theorem \ref{th:p1}) of $p_1(m)$ for $m > 1$ follow from the similar properties of $p_1(m)$ for $m\in(0,1)$.
\subsection{Differentiability of $p_1(m)$}
In this subsection we prove that the function $p_1(m)$ is a continuous and decreasing function for all $m>1$ and $p_1(m)$ is continuously differentiable at all points of the interval $m \in (1,+\infty)$ except for the set $\{m = \frac{k+1}{k}| k \in \N\}$ where its derivative is equal to $-\infty$. From functional equation (\ref{eq:g1_p1sym}) it immediately follows that the function $p_1(m)$ is a continuous and increasing function for all $m\in(0,1)$ and $p_1(m)$ is continuously differentiable at all points of the interval $m \in (0,1)$ except for the set $\{m = \frac{k}{k+1}| k \in \N\}$ where its derivative is equal to $+\infty$.  

Let $m>1$. From (\ref{eq:x1pi}) and (\ref{eq:g1_p1sym}) it follows that 
\[p_1(\frac{k+1}{k}) = \pi (k+1) \text{ for any } k \in \N.\]
Consider the partial derivatives
\begin{eqnarray}
&& \pdd{g_1(p,m)}{p} = \frac{m^2-1}{m} \sin{p} \sin \frac{p}{m},\label{dg1dp}\\
&& \pdd{g_1(p,m)}{m} =  \frac{m p \cos (p) \cos \left(\frac{p}{m}\right)+\left(p \sin (p)-m^2 \cos (p)\right) \sin \left(\frac{p}{m}\right)}{m^2}.\label{dg1dm}
\end{eqnarray}
We have $\pdd{g_1}{p}(\pi (k+1), \frac{k+1}{k}) = 0$ and $\pdd{g_1}{m}(\pi (k+1), \frac{k+1}{k}) = -\pi k<0$. Hence the plot of $p_1(m)$ has a vertical tangent line. We claim that $p_1(m)$ is smooth for $m > 1$, $m \notin \{\frac{k+1}{k}| k \in \N\}$. Indeed, this follows from the implicit function theorem and the inequality $\pdd{g_1}{p}|_{p=p_1(m)} \neq 0$. Let us prove the last inequality. The equation $\pdd{g_1}{p} = 0$ has two series of positive roots, namely $p^1(k_1) = \pi k_1$ and $p^2(k_2) = \pi k_2 m$ where $k_1,k_2 \in \N$. At these points we have 
\[g_1(p^1(k_1),m) = (-1)^{k_1+1} m \sin(\frac {k_1 \pi}{m}), \quad g_1(p^2(k_2),m) = (-1)^{k_2} m \sin(k_2 m \pi). \]
It can easily be checked that $g_1(p^i(k_i),m) \neq 0$ for $m \in (\frac{k+2}{k+1},\frac{k+1}{k})=:I(k)$, $\forall k \in \N$. This implies that $\pdd{g_1}{p}|_{p=p_1(m)} \neq 0$. From the implicit function theorem it follows that $p_1(m)$ is a continuously differentiable function on $I(k)$ and 
\begin{eqnarray*}
p_1'(m) = -\pdd{g_1(p,m)}{m}/\pdd{g_1(p,m)}{p}|_{p = p_1(m)} = -
\frac{p + m \cot p(-m + p \cot\frac{p}{m})}{m(m^2 -1)}|_{p = 
p_1(m)}.
\end{eqnarray*} 
Now we claim that $p_1(m)$ monotonically decreases for $m \in (1,2)$, $m \notin \{\frac{k+1}{k}| k \in \N\}$. Indeed, this follows from the inequality $p_1'(m)<0$. Let us prove this inequality. Since $g_1(p,m)|_{p = p_1(m)} = 0$, we get $m \cot p |_{p = 
p_1(m)} = \cot\frac{p}{m}|_{p = p_1(m)}$. Then $\sign(p_1'(m)) = -
\sign(f(p_1(m), m))$ where $f(p,m) = p - m \cot\frac{p}{m} + p \cot^2 \frac{p}{m}$. Consider the function $f$ as a polynomial of second degree with respect to $\cot\frac{p}{m}$. Since $m<2$ and $p_1(m)>\rho>1$ (see Theorem \ref{th:p1} e)), we see that the discriminant of the polynomial is negative, indeed $\mathrm{D}(f) = m^2 - 4 p^2 <0$. Thus $f(p_1(m), m)>0$ and $p_1'(m)<0$.       

Now we claim that the function $p_1(m)$ is continuous at the points $m = \frac{k+1}{k}$ for any $k \in \N$. Indeed, this means  existence of limits of $p_1(m)$ at the considered points and $\lim\limits_{m\to \frac{k+1}{k} \pm 0} p_1(m) = p_1(\frac{k+1}{k}) = \pi (k+1)$. Since $p_1(m)$ is monotonic and bounded for $m\neq \frac{k+1}{k}$ it follows that there exists finite limits $p_{\pm}(k) = \lim\limits_{m\to \frac{k+1}{k} \pm 0} p_1(m)$. Inequalities $p_{+}(k)<\pi(k+1)$, $p_{-}(k)<\pi(k+1)$ contradict to $p_1(\frac{k+1}{k}) = \pi (k+1)$ because if they are satisfied, then there exists a positive root $\tilde{p} < p_1(m)$ of the function $g_1(p,m)$, but by definition $p_1(m)$ is the minimal positive root. In such a way we proved that $p_{+}(k)\geq\pi(k+1)$, $p_{-}(k)\geq\pi(k+1)$. Further, the inequalities $p_{+}(k)>\pi(k+1)$, $p_{-}(k)>\pi(k+1)$ contradict to continuity of the curve $\{(p,m)|g_1(p,m)=0\}$ in a neighborhood of $(\pi (k+1), \frac{k+1}{k})$, but from the implicit function theorem this curve is continuous because    
$\pdd{g_1}{m}(\pi (k+1),\frac{k+1}{k}) = -m \frac{k+1}{k} \neq 0$.
Thus we proved that $p_{+}(k)=\pi(k+1)$, $p_{-}(k)=\pi(k+1)$. So $p_1(m)$ is a continuous function for any $m \in (1,2]$. 

Now we show that $p_1'(\frac{k+1}{k}) = -\infty$ for any $k \in \N$. For any $\mast = \frac{k+1}{k}$ we know the explicit value $\past = p_1(\mast) = \pi (k+1)$. We have $\pdd{g_1(\past,\mast)}{m}= -\pi k$ and $\pdd{g_1(\past,\mast)}{p}= 0$. This implies existence of the limit $\lim\limits_{m \to \mast}p_1'(m) = -\lim\limits_{{m \to \mast}}\pdd{g_1}{m}/\pdd{g_1,m)}{p}(p_1(m),m) = -\infty$. Therefore there exists $p_1'(\mast)=-\infty$.  

Thus we proved that $p_1(m)$ is continuous for any $m \in (1,2]$ and $p_1 \in C^1(\Omega)$ where $\Omega := \{m \in (1,2)|m \neq \frac{k+1}{k}, k \in \N\}$. Now let $m>2$. We have $\pdd{g_1}{p}|_{p = p_1(m)} \neq 0$. We claim that $p_1'(m)<0$. Indeed, this means that $f(p,m)>0$ for $p = p_1(m)$. In fact, if $\mathrm{D}(f)<0$, then $f(p,m)>0$; if $\mathrm{D}(f) \leq 0$, then $m\geq 2 p$ and we get $\cot\frac{p}{m} > q_2$ where $q_2$ is the greatest root of the function $f$ as a polynomial of second degree with respect to $\cot\frac{p}{m}$. Thus, we proved that $f(p_1(m),m)>0$. It follows that $p_1'(m)<0$. So $p_1(m)$ is a continuously differentiable decreasing function for $m \in (2, +\infty)$      

Finally we show that the function $p \mapsto g_1(p,m)$ changes sign at the point $p = p_1(m)$. If $m \notin \{\frac{k+1}{k}| k \in \N\}$ this follows from the inequality $\pdd{g_1}{p}|_{p = p_1(m)} \neq 0$. If $m = \mast = \frac{k+1}{k}$ we have $\pdd{g_1}{p}|_{p = p_1(m)} = 0$, $\pnder{g_1}{p}{2}|_{p = p_1(m)} = 0$, $\pnder{g_1}{p}{3}|_{p = p_1(m)} = -\frac{2(1+2 k)}{(k+1)^2}<0$. It follows that $g_1(p,\mast)$ changes sign at the point $\past = p_1(\mast)$.     
\section{Study of the function  $p_2(m)$}
\label{sec:g2}
In this section we prove a two-sided estimate of a minimal positive root of the equation $g_2(p, m)=0$, where
\begin{equation}\label{eq:g2dubl}
g_2(p,m)=m p \cos \left(\frac{p}{m}\right)\sin p - (p\cos p + (m^2 -1)\sin p)\sin \left(\frac{p}{m}\right).
\end{equation}
Our aim is to find or estimate as accurately as possible for any $m\in(0,1)\cup(1,+\infty)$ the value of $p_2(m)$ defined in (\ref{eq:p2}). The main results on this problem are summarized in Theorem \ref{th:p2}.
\begin{theorem}\label{th:p2}
For any $m \in (0,1) \cup (1,+\infty)$ the function $g_2(p,m)$ has a
minimal positive root $p_2(m)$ satisfying the following properties:
\begin{enumerate}
\item[a)] The function $p_2(m)$ is continuous for $m\in(0,1)$, coincides with $\frac{\pi m}{1-m}$ at the points $\{m = \frac{k}{k+1}|k \in \N\}$ and at points $m = \mast_{1,k}$, where $\frac{1+2k}{3+2k}-\frac{2}{15+40k+32k^2+8 k^3}<\mast_{1,k} <\frac{1+2k}{3+2k}$; is continuous for $m\in(1,+\infty)$, coincides with $\frac{\pi m}{m-1}$ at the points $\{m = \frac{k+1}{k}|k \in \N\}$  and at points $m = \mast_{2,k}$, where $\frac{3+ 2k}{1+2k}<\mast_{2,k}<\frac{3+2k}{1+2k}+\frac{2}{1+8k(1+k^2)}$.
\item[b)] The function $p_2(m)$ is continuously differentiable at all points of the set $(1,\frac12) \cup (2, +\infty)$ and the following sets:
\begin{eqnarray*}
m \in \left[\frac{k}{1+k},\frac{1+2k}{3+2k}-\frac{2}{15+40k+32k^2+8 k^3} \right]\cup \left[\frac{1+2k}{3+2k},\frac{1+k}{2+k}\right] \cup \\
\cup \left[\frac{k+2}{k+1},\frac{3+ 2k}{1+ 2k}\right]\cup
 \left[\frac{3+2k}{1+2k}+\frac{2}{1+8k(1+k^2)},\frac{k+1}{k}\right]
\end{eqnarray*}
for all $k \in \N$. Out of these sets there exist values  $m=\bar{m}_{1,k}$, $\frac{1+2k}{3+2k}-\frac{2}{15+40k+32k^2+8 k^3}<\mast_{1,k}<\bar{m}_{1,k}<\frac{1+2k}{3+2k}$, and $m=\bar{m}_{2,k}$, $\frac{3+ 2k}{1+2k}<\bar{m}_{2,k}<\mast_{2,k}<\frac{3+2k}{1+2k}+\frac{2}{1+8k(1+k^2)}$, where the derivative $p_2'(m)$ is equal to $+\infty$ and $-\infty$ respectively.
\item[c)] For $m \in (0,1)$ the function $p_2(m)$ has the following lower and upper bounds:
\begin{equation}\label{p2hypm01}
\begin{cases}
5.7 m < p_2(m) < 2 \pi m \text{ for } m<\frac12,\\ 
\frac{\pi m}{1-m} - \frac{m}{1-m}\arcsin\left(\frac{1-m}{1+m}\right) <p_2(m)<\frac{\pi m}{1-m} + \frac{m}{3 m -1}
\text{ for } m \geq \frac12. 
\end{cases}
\end{equation}
For $m >1$ the function $p_2(m)$ has the following lower and upper bounds:
\begin{equation}\label{p2hypmgr1}
\begin{cases}
\frac{\pi m}{m-1} - \frac{m}{m-1}\arcsin\left(\frac{m-1}{m+1}\right) <p_2(m)<\frac{\pi m}{m-1} + \frac{m}{3-m}\text{ for } m\leq 2, \\ 
5.7< p_2(m) < 2 \pi \text{ for } m > 2.
\end{cases}
\end{equation}
\item[d)] For $m \in (\frac12,1)$ the plot of $p = p_2(m)$ ''wraps'' the hyperbole $p = \frac{\pi m}{1-m}$ as follows:
\begin{equation}\label{p2hypm01prec}
\begin{cases}
p_2(m)<\frac{\pi m}{1-m} \text{ for } m \in \left(\frac{k}{1+k},\frac{1+2k}{3+2k}-\frac{2}{15+40k+32k^2+8 k^3} \right],\\
p_2(m)>\frac{\pi m}{1-m} \text{ for } m \in \left[\frac{1+2k}{3+2k},\frac{1+k}{2+k} \right),
\quad k \in \N.
\end{cases}
\end{equation}
For $m \in (1,2)$ the plot of $p = p_1(m)$ ''wraps'' the hyperbole $p = \frac{\pi m}{m-1}$ as follows:
\begin{equation}\label{p2hypmgr1prec}
\begin{cases}
p_2(m)>\frac{\pi m}{m-1} \text{ for } m \in \left(\frac{k+2}{k+1},\frac{3+ 2k}{1+ 2k} \right],\\ 
p_2(m)<\frac{\pi m}{m-1} \text{ for } m \in \left[\frac{3+2k}{1+2k}+\frac{2}{1+8k(1+k^2)},\frac{k+1}{k} \right),
\quad \forall k \in \N.
\end{cases}
\end{equation}
\item[e)] At the point $p = p_2(m)$ the function $p \mapsto g_2(p,m)$ changes its sign.
\end{enumerate} 
\end{theorem}
Figure \ref{fig:p2bounds} shows the plot of $p_2(m)$ with two-sided estimates (dashed lines).
\begin{figure}
\includegraphics[width=\textwidth]{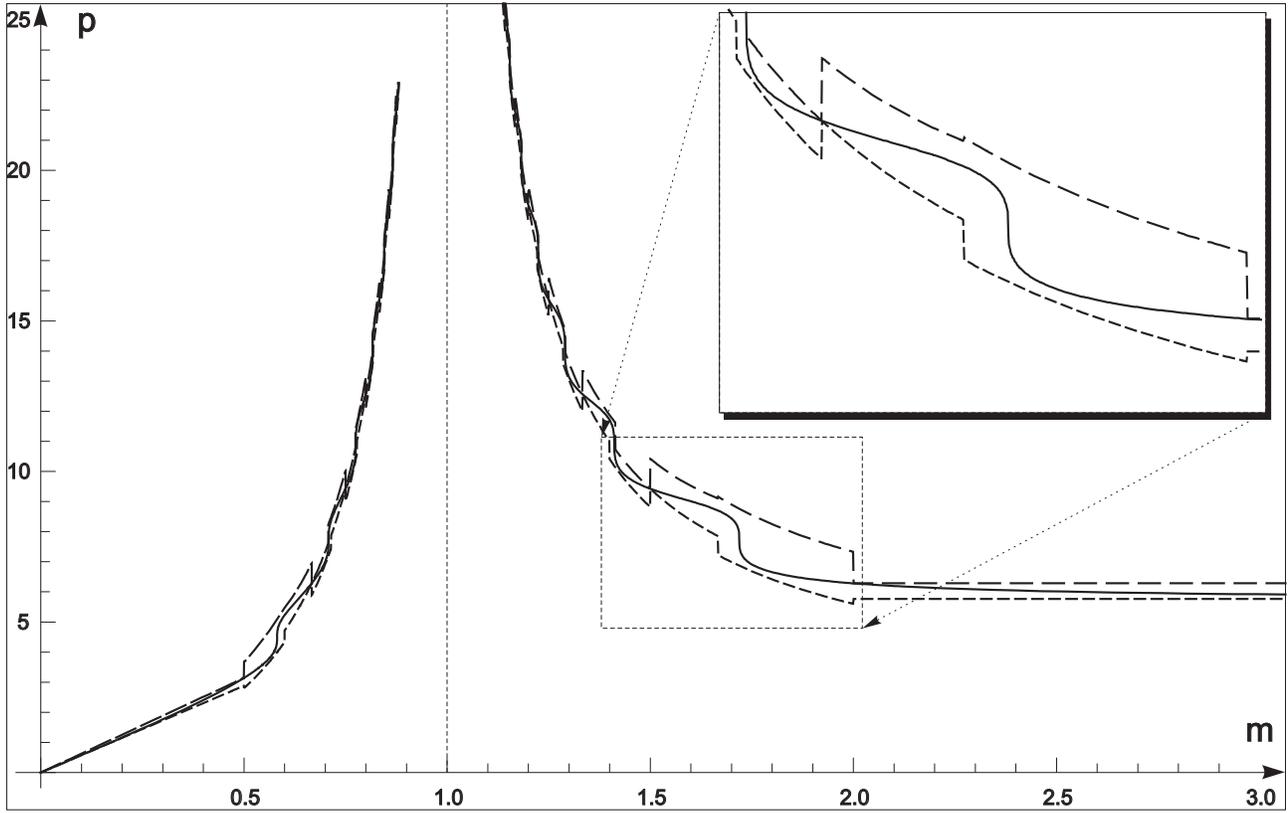}
\caption{Plot of the function $p_2(m)$ with two-sided estimates}
\label{fig:p2bounds}
\end{figure}
\begin{remark}
From (\ref{p2hypm01}) and (\ref{p2hypmgr1}) it follows that $p_2(m)$ tends to $+\infty$ as $m\to 1 \pm 0$.
\end{remark}
\subsection{Two-sided estimate of $p_2(m)$ for $0<m<\frac12$}\label{subsec:p1m012}
Let us prove that for $m \in (0, \frac12)$ the equation $g_2(p, m)=0$ (see (\ref{eq:g2dubl})) has a root in the interval $0<p< 2 \pi m$.
First write the equation $g_2(p,m)=0$ in the equivalent form
\begin{equation}
\label{eq:g2sym}
m p \cos \left(\frac{p}{m}\right)\sin p = (p\cos p + (m^2 -1)\sin p)\sin \left(\frac{p}{m}\right).
\end{equation}
Since $0<m<\frac12$, we see that in the interval $0<p< 2 \pi m$ the product $\sin p \sin \left(\frac{p}{m}\right)$ vanishes only at $p=\pi m$ but it can easily be checked that this point is not a root of equation (\ref{eq:g2sym}). Therefore, dividing both sides of (\ref{eq:g2sym}) by the product $\sin p \sin(\frac{p}{m})$, we get the equivalent equation
\begin{align*}
m p \cot(\frac{p}{m}) = p \cot p + m^2-1 \iff m p \cot(\frac{p}{m})-m^2	= p \cot p -1 \iff \\
\iff \ \frac{m p \cot(\frac{p}{m})-m^2}{p^2}	= \frac{p \cot p -1}{p^2} \ \iff \ \frac{\frac{p}{m} \cot(\frac{p}{m})-1}{(\frac{p}{m})^2}	= \frac{p \cot p -1}{p^2}.
\end{align*}
Let $G(x) = \frac{1-x\cot x}{x^2}$. We have
$$g_2(p,m)=0 \iff G(\frac{p}{m}) = G(p) \text{ for } 0<p<2 \pi m.$$ 
Put $x = \frac{p}{m}$. Rewrite the last equation as  
\begin{equation}
\label{eq:G}
G(x) = G(m x), \qquad 0<x<2 \pi.
\end{equation}
Since the function $G(x)$ decreases over the intervals $0<x<\pi$ and $\pi<x<2 \pi$, we see that equation (\ref{eq:G}) has no roots in the interval $0<x<\pi$. Moreover, since $\lim_{x \to 0} G(x) = \frac{1}{3}$, it has no roots in the interval $\pi<x<\rho_2$, where $\rho_2$ is the root of the equation $G(x) = \frac13$ ($\rho_2= 5.7634\dots$). On the other hand, from the same reasons related to monotonicity of $G$ it follows that for any $x \in(\rho_2, 2 \pi)$ there exists a unique value $m=m(x) \in (0,\frac12)$ such that there holds equality (\ref{eq:G}). 

We claim that this dependence $m(x)$ is differentiable and has a positive derivative. Indeed, we must prove that 
\begin{equation}
\label{eq:dmdx}
m'(x) = \frac{G'(x) - m(x) G'(m(x) x)}{x G'(m(x) x)}>0. 
\end{equation}
Inequality (\ref{eq:dmdx}) follows from the following inequality:
\begin{align*}
G'(x) - m G'(m x) >0 \text{ for } G(m x) - G(x) = 0.
\end{align*} 
Let us prove it. We have
\begin{align*}
G'(x) = \left(\frac{1}{x^2}- \frac{\cot x}{x} \right)' = \frac{1 + \cot^2 x}{x} - \frac{1}{x^3} - \frac{G(x)}{x},
\end{align*}
\begin{multline*}
m G'(m x) = m \left(\frac{1+\cot^2(m x)}{m x}- \frac{1}{(m x)^3} - \frac{G(m x)}{m x} \right)= \frac{1 + \cot^2( m x)}{x} - \frac{1}{m^2 x^3} - \frac{G(m x)}{x},
\end{multline*}
where $G'(m x)$ is the derivative $G'(x)$ taken at $m x$. Since $\frac{G(m x)}{x} = \frac{G(x)}{x}$ we must prove the following inequality: 
\begin{multline*}
\frac{\cot^2(m x)}{x} - \frac{\cot^2(x)}{x}< \frac{1}{m^2 x^3} -\frac{1}{x^3} \Leftrightarrow \cot^2(m x) - \cot^2(x)< \frac{1}{m^2 x^2} -\frac{1}{x^2}.
\end{multline*}
Since $\cot t <\frac{1}{t}$ for any $t \in (0,\pi)$, we have $\cot^2(m x)<(m x)^{-2}$. Further, since for any $x \in (\rho_2,2\pi)$ we have $\cot x < -1$, it follows that $\cot^2 x - x^{-2}>1-\rho_2^{-2}>0$. Thus, we proved the required inequality. 

Now, since the dependence $m(x)$ is differentiable and has a positive derivative we can use the inverse function theorem and conclude that for any $m \in (0, \frac12)$ there exists the differentiable and increasing function $x(m) \in (\rho_2, 2 \pi)$ such that $x(m)$ is a root of equation (\ref{eq:G}). Thus, we see that $p_2(m) = m x(m)$ is a differentiable increasing function for $m \in (0, \frac12)$, which admits the following estimate:   
\begin{equation}
\label{eq:p2m012}
\rho_2 m < p_2(m)<2\pi m, \qquad 0 < m < \frac12.
\end{equation}  

In addition, we claim that for any $m \in (0,\frac12)$ the function $p_2(m)$ exceeds both $p_0(m) = \frac{\pi m}{1-m}$ and $p_1(m)$ (see Theorem~\ref{th:p1}). Indeed, for $0<m\leq\frac13$ this follows from the inequalities $\frac{\pi m}{1-m}\leq p_1(m)\leq\frac{3 \pi m}{2}$ and $p_2(m)>\rho_2 m >\frac{3\pi m}{2}$. For $\frac13<m<\frac12$ we have $p_1(m) \leq \frac{\pi m}{1-m}$ and it remains to prove that 
\begin{equation}
\label{eq:p2hyp}
\frac{\pi m}{1-m}<p_2(m), \qquad \frac13<m<\frac12.
\end{equation} 
In fact, it can easily be checked that for any $m\neq 1$ there holds the equality $\cot(\frac{\pi}{1-m})=\cot(\frac{\pi m}{1-m})$ and for any $m \in (\frac13,\frac12)$ there holds the inequality $\cot(\frac{\pi}{1-m})<0$. By definition, put $A= -\cot(\frac{\pi}{1-m})=-\cot(\frac{\pi m}{1-m})$ and $G_1(x) = G(x) - G(m x)$. The function $G_1(x)$ increases over the interval $\pi<x<2 \pi$. It follows from definition of $p_2(m)$ that $G_1(\frac{p_2(m)}{m})=0$. Since $G_1(x)$ increases, we have    
\begin{eqnarray*}
(\ref{eq:p2hyp})\iff G_1(\frac{\pi}{1-m})<0 \iff G(\frac{\pi}{1-m})<G(\frac{\pi m }{1-m})\iff \\
\iff \frac{1 - \frac{\pi}{1-m}\cot(\frac{\pi}{1-m})}{(\frac{\pi}{1-m})^2}<\frac{1 - \frac{\pi m}{1-m}\cot(\frac{\pi m}{1-m})}{(\frac{\pi m}{1-m})^2} \iff \\
\iff 1+ \frac{\pi A}{1-m}<\frac{1+ \frac{\pi m}{1-m} A}{m^2} 
\iff 1 + \frac{\pi A}{1-m}< \frac{1}{m^2} + \frac{\pi A}{1-m}\frac{1}{m}.	
\end{eqnarray*}
The last inequality is obvious since $1<\frac{1}{m^2}$, $\frac{\pi A}{1-m}<\frac{\pi A}{1-m}\frac{1}{m}$. Thus, inequality (\ref{eq:p2hyp}) is proved.  
\subsection{Reduction of $g_2(p,m)=0$ to a simpler form} \label{subsec:p2x2}
In this subsection we reduce the problem of finding (estimation) of $p_2(m)$ for $m \in [\frac12,1)$ to finding (estimation) of the minimal positive root $x_2(s)$ of a simpler equation $\tilde{g}_2(x,s)=0$ for $s\geq 3$, where $\tilde{g}_2(x,s)=4 s(\cos x - \cos(s x))- x (s^2-1)(s \sin x + \sin (s x))$.

It can easily be checked that
\begin{eqnarray*}
g_2(p,m)= &&-(m^2 -1)\sin p\sin(\frac{p}{m}) +  p(\frac{m+1}{2}(\sin p\cos(\frac{p}{m}) -\\
&&-\cos p\sin(\frac{p}{m})) + \frac{m-1}{2}(\sin p\cos(\frac{p}{m}) + \cos p\sin(\frac{p}{m}))). 
\end{eqnarray*}
Transform the first summand:
\begin{eqnarray*}
&\left(m^2 -1\right)\sin p\sin\left(\frac{p}{m}\right) = \frac{m^2 - 1}{2}\left(\cos p\cos\left(\frac{p}{m}\right)+\sin p\sin\left(\frac{p}{m}\right)-\cos p\cos\left(\frac{p}{m}\right)+\sin p\sin\left(\frac{p}{m}\right)\right) = \\
&= \frac{\left(m-1\right)\left(m+1\right)}{2}\left(\cos\left(p \left(1-\frac{1}{m}\right)\right)-\cos\left(p \left(1+\frac{1}{m}\right)\right)\right)=\\
&= -\frac12 \left(1+\frac{1+m}{1-m}\right)^{-2} \frac{4 \left(1+m\right)}{1-m}\left(\cos\left(\frac{p\left(1-m\right)}{m}\right)-\cos\left(\frac{1+m}{1-m} \frac{p\left(1-m\right)}{m}\right)\right).
\end{eqnarray*} 
Further, transform the second summand:
\begin{eqnarray*}
&p\left(\frac{m+1}{2}\sin\left(p-\frac{p}{m}\right) + \frac{m-1}{2}\sin\left(p+\frac{p}{m}\right)\right) = \frac{p}{2}\left(m-1\right)\left(\frac{m+1}{m-1}\sin\left(p-\frac{p}{m}\right)+ \sin\left(\frac{1+m}{1-m} \frac{p\left(1-m\right)}{m}\right)\right)=\\
&=\frac12 \left(1+\frac{1+m}{1-m}\right)^{-2}\left(-1+\left(\frac{1+m}{1-m}\right)^2\right) 
\left(\frac{-p\left(1-m\right)}{m}\right)\left(\frac{1+m}{1-m}\sin\left(\frac{p\left(1-m\right)}{m}\right)+ \sin\left(\frac{1+m}{1-m} \frac{p\left(1-m\right)}{m}\right)\right).
\end{eqnarray*} 
Let us introduce a parameter $s=\frac{1+m}{1-m}$ and a new variable $x=p\frac{1-m}{m}$. Now express $g_2(p,m)$:
\begin{eqnarray*}
&g_2(p,m)=\frac{1}{2(1+s)^2} \tilde{g}_2(x,s),
\text{ where } \tilde{g}_2(x,s) = 4 s\left(\cos x - \cos\left(s x\right)\right)- x \left(s^2-1\right)\left(s \sin x + \sin \left(s x\right)\right).
\end{eqnarray*}
Thus, we reduce the problem of finding (estimation) of $p_2(m)$ for $m \in (\frac12,1)$ to finding (estimation) of the following function: 
\begin{equation*}
x_2(s) = \min\{x>0|\tilde{g}_2(x,s)=0\} \text{ for } s\geq 3.
\end{equation*}
In Subsections \ref{subsec:x2sgeq3} -- \ref{subsec:x2sgeq3moreaccurate} we obtain a two-sided estimate of $x_2(s)$. Returning to the original variables $p = x \frac{s-1}{2}$, $m = \frac{s-1}{s+1}$ we obtain the statement of Theorem \ref{th:p2}.
\subsection{Two-sided estimate of $x_2(s)$ for $s\geq 3$} \label{subsec:x2sgeq3}
In this subsection we prove that $x_2(s)$ admits the following lower and upper bounds:
\begin{equation}\label{eq:g2_5}
\pi - \arcsin\frac{1}{s}<x_2(s)<\pi + \frac{1}{s-2} \text{ for } s\geq3, 
\end{equation} 
and $x_2(s)$ is a unique root of $\tilde{g}_2(x, s)$ on the interval $(\pi - \arcsin\frac{1}{s},\pi + \frac{1}{s-2})$.
\begin{proposition}\label{propos:g2_1}
There holds the inequality
\begin{equation}
\label{eq:g2_6}
x_2(s)<\frac{3 \pi}{2}.
\end{equation}
\end{proposition}
\begin{proof}
Indeed, (\ref{eq:g2_6}) follows from the fact that $\tilde{g}_2(x, s)$ has zero derivatives up to the fifth order at the point $x=0$ and the sixth derivative is negative. In fact, $\pnder{\tilde{g}_2(x,s)}{x}{6}|_{x=0} = - 2 s (s^2-1)^3<0$  for all $s>1$. It means that the function $\tilde{g}_2(x, s)$ is negative in some right half-neighborhood of the point $x=0$. But at the point $x = \frac{3 \pi}{2}$ this function is positive, since  
\begin{multline*}
\tilde{g}_2(\frac{3 \pi}{2}, s) = - 4s \cos\left(\frac{3 \pi s}{2}\right) - \frac{3 \pi}{2}\left(s^2-1\right)\left(\sin\left(\frac{3 \pi s}{2}\right)-s\right)>-4 s - \frac{3 \pi}{2}\left(s^2-1\right)\left(1-s\right)=\\
=\frac{3 \pi}{2}\left(s^3-s^2-s\left(1+ \frac{8}{3 \pi}\right)+1 \right)>s^2-s-2>0 \text{ for any } s\geq 3.
\end{multline*}
Hence, we see that the continuous function $\tilde{g}_2(x, s)$ changes its sign over the interval $(0, \frac{3 \pi}{2}]$. This implies that $\tilde{g}_2(x, s)$ has a root in the interval $(0,\frac{3 \pi}{2})$. Thus, inequality (\ref{eq:g2_6}) is proved. 
\end{proof}
Now we rewrite $\tilde{g}_2(x, s)$ in the following form:
$$\tilde{g}_2(x, s) = 4 s \left(\cos x - \cos \left(s x\right)\right) + x H(x,s),$$
where $H(x,s) = -(s^2-1) h(x,s)$ and $h(x,s) = s \sin x + \sin \left(s x\right)$.

Let $J(x,s)=\frac{\tilde{g}_2(x, s)}{H(x,s)}$. It can easily be shown that
$$\pder{J(x,s)}{x}=\frac{\tilde{g}(x, s)^2}{h(x,s)^2}\geq0, \text{ where } \tilde{g}(x, s) \text{ was defined in \ref{subsectildeg1}}.$$
Thus we can see that if $\tilde{g}(x, s)$ and $h(x,s)$ do not vanish at the same time (this case is considered separately in Remark \ref{rem:g2_1}), then $J(x,s)$ increases at all points, where $h(x,s)\neq 0$, and has vertical asymptotes at the points $\tilde{x}(s)=\{x|h(x,s)=0\}$. The function $J(x,s)$ has the following properties: 1) $\lim_{x \to 0} J(x,s) = 0$, 2)$\lim_{x \to \tilde{x}_k(s) -0} J(x,s) = +\infty$, 3)$\lim_{x \to \tilde{x}_k(s) +0} J(x,s) = -\infty$, where $\tilde{x}_k(s)$ is the $k$-th positive root of $h(x,s)$. Thus, we conclude that the function $J(x,s)$ (and hence the function $\tilde{g}_2(x,s)$) has no roots in the interval $(0,\tilde{x}_1(s))$ but it has a unique root in the interval $(\tilde{x}_1(s), \tilde{x}_2(s))$. Hence, we arrive to the problem of finding (estimation) of $\tilde{x}_1(s)$ and $\tilde{x}_2(s)$. We study this problem in Proposition  \ref{stat:g2_1} (see below). Combining the estimations $x_2(s) \in (0,\frac{3 \pi}{2})$ and $\tilde{x}_1(s)<x_2(s)<\tilde{x}_2(s)$ with Proposition \ref{stat:g2_1} we get the following two-sided estimation of $x_2(s)$:
\begin{equation}
\label{eq:g2_11}
\pi-\arcsin\frac{1}{s}<x_2(s)<\frac{3 \pi}{2}.
\end{equation} 
\begin{remark}
\label{rem:g2_1}      
Consider the case when $\tilde{g}_2(x,s)$ and $H(x,s)$ have a common root for $0<x<\frac{3 \pi}{2}$. This holds if and only if $h(x,s)=0$ and $\cos x - \cos s x =0$. In Subsection \ref{sec_4_3} (see (\ref{eq:criticalPoints_g1})) we found the roots of the second equation, namely $x \in \{\frac{2 \pi n}{s \pm 1}|n\in \N\}$. Let us find the values of $s$, for which these roots are the roots of $h(x,s)$ simultaneously, in other words $s \sin\frac{2 \pi n}{s \pm 1}+ \sin \left(s \frac{2 \pi n}{s \pm 1}\right) = 0$. The last equation is equivalent to $(s \mp 1) \sin \frac{2 \pi n}{s \pm 1}=0 \Leftrightarrow \frac{2 n}{s \pm 1} = k,$ $k \in \N$. Thus, $s = \frac{2 n}{k} \mp 1$. So, we got that $\tilde{g}_2(x,s)$ and $H(x,s)$ vanish simultaneously if and only if $x= \pi k$ and $s = \frac{2 n}{k} \mp 1$, where $k, n \in \N$. It is clear that only for $k = 1$ and $s = 2 n + 1$ this solution satisfies the limitations $x \in (0, \frac{3 \pi}{2})$ and $s\geq 3$. Thus, we get 
$$\tilde{g}_2(x,s)=H(x,s)=0 \quad \Leftrightarrow \quad x= \pi, \ s = 2 n + 1, \ n \in \N. $$
Since $J(0,s)=0$, $J(x,s)$ increases over the interval $x \in (0, \pi)$, and $h(x,s)>0$ for $s=2 n +1$ in this interval (see Proposition \ref{stat:g2_1}), we see that $J(x,s)>0$ in this interval. This implies that for $s=2 n +1$ the function $\tilde{g}_2(x,s)$ has no roots for $x \in (0, \pi)$ and $x = \pi$ is the minimal positive root. Thus, we get 
\begin{equation}
\label{eq:g2_12}
x_2(s) = \pi \text{ for } s = 2 n + 1, \quad n \in \N.
\end{equation}
\end{remark}
In the following proposition we estimate the first and the second positive roots of the function $h(x,s) = s \sin x + \sin \left(s x\right)$. 
\begin{proposition}\label{stat:g2_1}
The functions $$\tilde{x}_1(s)=\min \left\{x>0| h(x,s)=0 \right\}$$ and $$\tilde{x}_2(s)= \min \left\{x>\tilde{x}_1(s)|h(x,s)=0\right\}$$ have the following two-sided estimates for $s\geq 3$:
$$\tilde{x}_1(s) \in \left[\pi-\arcsin \frac{1}{s}; \pi+\arcsin \frac{1}{s}\right]= I_s^1,$$
$$\tilde{x}_2(s) \in \left[2\pi - \arcsin \frac{1}{s}; 2\pi + \arcsin \frac{1}{s}\right] = I_s^2.$$
This estimate for $\tilde{x}_1(s)$ can be improved as follows:
\begin{align*}
\begin{cases}
\tilde{x}_1(s)=\pi, \text{ for } s\in \N, \\
\tilde{x}_1(s)\in \left[\pi - \arcsin \frac{1}{s}; \pi\right), \quad s\in(2k+1, 2k+ 2),\\
\tilde{x}_1(s)\in \left(\pi; \pi +\arcsin \frac{1}{s}\right],   \quad s \in (2k+2, 2k+3), \quad k\in \N.
\end{cases}
\end{align*}
\end{proposition}
\begin{proof}
First, we show that there exist roots of $h(x,s)$ in $I_s^1$ and $I_s^2$ 
\begin{align}
h\left(\pi \mp \arcsin \frac{1}{s}, s\right)=\pm1+ \sin\left(\pi \mp s\arcsin \frac{1}{s}\right)\Rightarrow \nonumber\\
h\left(\pi - \arcsin \frac{1}{s}, s\right)\geq 0, \text{   }  h\left(\pi + \arcsin \frac{1}{s}, s\right)\leq 0,\label{eq:g2_7}\\
h\left(2\pi \mp \arcsin \frac{1}{s}, s\right)=\mp1+\sin\left(2\pi s\mp s \arcsin \frac{1}{s}\right)\Rightarrow \nonumber\\
h\left(2\pi - \arcsin \frac{1}{s}, s\right)\leq 0, \text{   } h\left(2\pi + \arcsin \frac{1}{s}, s\right)\geq 0.\label{eq:g2_8} 
\end{align}
Further, if $\arcsin\left(\frac{1}{s}\right)<x <\pi-\arcsin\frac{1}{s}$, then $h(x,s)>0$ because in this interval we have $s \sin x > 1$. Now we claim that $h(x,s)>0$ for $0<x\leq\arcsin\frac{1}{s}$. Indeed, this follows since $h(x,s)$ is increasing over this half-interval and from equality $h(0,s)=0$. The last equality can easily be checked by direct calculation. Let us prove that $h(x,s)$ increases over the half-interval $x \in \left(0,\arcsin\frac{1}{s}\right]$. To prove this we consider the derivative $h'_s(x,s)=s(\cos x +\cos sx)$ and find all positive critical points   
\begin{eqnarray}
h'_x(x,s)=0 & \Leftrightarrow \cos x = - \cos sx \Leftrightarrow 
\begin{cases}
x = s x +\pi - 2\pi n_1,\quad n_1 \in \N\\
-x = s x +\pi - 2\pi n_2,\quad n_2 \in \N
\end{cases}
\Leftrightarrow \nonumber  \\
& \Leftrightarrow x \in\{\frac{2 n-1}{s\pm1} \pi|n \in \N\}. \label{eq:g2_9}
\end{eqnarray}
Since for any $s\geq 3$ we have $\min \left\{ \frac{2n-1}{s\pm1}\pi\right\}=\frac{\pi}{s\pm1}>\frac{\pi}{2s}>\arcsin \frac{1}{s}$, we see that the half-interval  $x \in \left(0,\arcsin\frac{1}{s}\right]$ does not contain any point of form (\ref{eq:g2_9}). Moreover, $h'_x(0,s)=2 s>0$ and we proved that $h(x,s)$ is increasing.  

So we proved that $h(x,s)>0$ for $0<x<\pi-\arcsin\frac{1}{s}$. Now we claim that there exists a unique root, denoted by $\tilde{x}_1(s)$, of the function $h(x,s)$ in the interval $\pi-\arcsin \frac{1}{s}\leq x \leq \pi + \arcsin \frac{1}{s}$. To prove this we divide the set $s\geq 3$ into the three subsets $s=k+2$, $s\in(2k+1, 2k+2)$, and $s\in(2k+2, 2k+3)$, where $k\in \N$.
\begin{enumerate}
\item Let $s=k+2$, $k\in \N$. We have $h(\pi,s)=0$. Uniqueness of the root $\tilde{x}_1(s) = \pi$ for $\pi-\arcsin \frac{1}{s}\leq x \leq \pi+\arcsin \frac{1}{s}$ follows from monotonicity of $h(x,s)$ w.r.t. $x$ in the intervals $\pi - \frac{\pi}{2s}<x<\pi$ and $\pi <x< \pi+\frac{\pi}{2s}$. Indeed, if the first interval contained a critical point (\ref{eq:g2_9}) then there would exist $n \in \N$ such that $(k+2)\pm 1 - \left(\frac{1}{2}\pm \frac{1}{2(k+2)}\right)<2n-1<(k+2)\pm 1$, but this contradicts to the following obvious inequality:  
\begin{equation}
\label{eq:g2_10}
\forall k \in \N: 0< \left(\frac{1}{2}\pm \frac{1}{2(k+2)}\right)<1.
\end{equation}
If the interval $\left(\pi; \pi+\frac{\pi}{2s}\right)$ contained a critical point (\ref{eq:g2_9}) then there would exist $n \in \N$ such that $(k+2)\pm1< 2n-1<(k+2)\pm 1+\left(\frac{1}{2}\pm \frac{1}{2(k+2)}\right)$, but this contradicts (\ref{eq:g2_10}).
\item Let $s\in(2k+1, 2k+2)$, $k\in \N$. We have $h(\pi, s)=\sin \pi s <0$. Combining this with (\ref{eq:g2_7}) we see that the function $h(x, s)$ has a root in the half-interval $\pi-\arcsin\left(\frac{1}{s}\right)\leq x<\pi$. Uniqueness of the root $\tilde{x_1}(s)$ follows from the constant sign of the derivative $h'_x(x,s)$ that is equivalent to absence of its roots. Indeed, let us show that the interval $\left(\pi -\frac{\pi}{2s}, \pi\right) \supset \left[\pi-\arcsin\frac{1}{s}, \pi\right)$ does not contain any critical point (\ref{eq:g2_9}). Assume the converse, then there exists $n \in \N$ such that $1-\frac{1}{2s}< \frac{2n-1}{s\pm 1}<1$. Since $2k+1< s < 2k+2$ we have $(2k+1\pm 1)\left(1- \frac{1}{2(2k+1)}\right)< 2n-1< 2k+2\pm 1$. Further, since $0< \frac{2k+1\pm 1}{2(2k+1)}=\frac{1}{2}\pm \frac{1}{2(2k+1)}<1$, $2n-1 \in \N$ and $2k+1\pm 1 \in \N$ we have $2n-1=2k+1\pm 1$. Now we see that no number $n$ satisfies the last equality, since $2n-1$ is odd and $2k+1\pm 1$ is even. This contradiction proves that the function $h(x,s)$ has a unique root for $\pi -\arcsin \frac{1}{s}\leq x<\pi$.
Further, we show that $h(\pi, s)<0$ for $\pi\leq x \leq \pi+\arcsin \left(\frac{1}{s}\right)$. It is sufficient to show that $h(x,s)<0$ at the end points of the considered segment and at all critical points inside it. At the left end we have $h(\pi, s) < 0$. At the right end we have $h\left(\pi+ \arcsin \frac{1}{s}, s\right)= -1 + \sin \left( \pi s + s \arcsin \frac{1}{s}\right)<0$, since from $(2k+1)\pi < s \pi + s \arcsin \frac{1}{s}< (2k+2) \pi + \frac{\pi}{2}$ it follows that  $\sin \left( s \pi + s \arcsin \frac{1}{s}\right)<1$. Now find critical points (\ref{eq:g2_9}) lying on the interval $\left(\pi, \pi+ \frac{\pi}{2s}\right) \supset \left(\pi, \pi + \arcsin \frac{1}{s}\right)$. It is easy to check that the inequalities $$\left(2k+1 \pm 1\right)< 2n-1< \left(2k+ 2 \pm 1\right) \left( 1 + \frac{1}{2(2k+1)}\right),\quad n, \ k \in \N$$ are satisfied only for the value $n = k +2$ if we select the sign ''plus'', and for $n = k+1$ if we select the sign ''minus''. Hence the considered interval contains only two critical points, namely $\frac{2k+3}{s+1}\pi$ and $\frac{2k+1}{s-1}\pi$. We can estimate the value of $h$ at these points as follows:
\begin{eqnarray*}
h\left(\frac{2k+3}{s+1}\pi, s \right) = &&s \sin \left(\frac{2k+3}{s+1} \right) + \sin \left(\frac{2k+3}{s+1}\pi s \right)=\\
= &&s \sin \left(\frac{2k+3}{s+1}\pi \right) + \sin \left(\left(2k+3\right)\pi - \frac{2k+3}{s+1}\pi\right) =\\
= &&(s+1) \sin \frac{2k+3}{s+1}\pi <0, \text{ since }\pi < \frac{2k+3}{s+1}< \pi + \frac{\pi}{2s}< 2\pi, \\
h\left(\frac{2k+1}{s-1}\pi, s \right) = &&s \sin \left(\frac{2k+1}{s-1}\pi\right)+ \sin \left(\frac{2k+1}{s-1}\pi s \right)= \\
=&&s \sin \left(\frac{2k+1}{s-1}\pi \right)+ \sin \left(\left(2k+1 \right) \pi + \frac{2k+1}{s-1}\pi \right)= \\
=&&(s-1) \sin \frac{2k+1}{s-1} \pi< 0, \text{  since } \pi < \frac{2k+1}{s-1} < \pi + \frac{\pi}{2s}< 2 \pi.
\end{eqnarray*}
Thus we proved that $h(x,s)<0$ for $\pi \leq x \leq \pi + \arcsin \frac{1}{s}$, therefore $h$ has no roots in this interval.
\item Let $s \in (2k+2, 2k+3)$, $k\in \N$. We claim that $h(x,s)>0$ if $\pi - \arcsin \left(\frac{1}{s}\right)\leq x \leq \pi$. It is sufficient to show that $h(x,s)>0$ at the end points of the considered segment and at all critical points inside it.  At the left end we have $h\left(\pi - \arcsin \frac{1}{s}, s\right)= 1+ \sin \left(\pi s - s \arcsin \frac{1}{s}\right)>0$, since from $(2k+2)\pi -\frac{\pi}{2}< \pi s - s \arcsin \frac{1}{s}< (2k+3)\pi$ it follows that $\sin\left(\pi s - s \arcsin \frac{1}{s}\right)> -1$. At the right end we have $h(\pi, s)= \sin \pi s >0$. It is easy to check that the inequalities $$(2k+2 \pm 1)\left(1- \frac{1}{2(2k+2)}\right)< 2n-1< 2k+3\pm 1, \quad n, \ k \in \N$$ are satisfied only for the value $n = k +2$ if we select the sign ''plus'', and for $n = k+1$ if we select the sign ''minus''. Hence the interval $\left(\pi- \frac{\pi}{2s},\pi\right) \supset \left[\pi - \arcsin \frac{1}{s}; \pi \right)$ contains only two critical points (\ref{eq:g2_9}), namely  $\frac{2k+3}{s+1}\pi$ and $\frac{2k+1}{s-1}\pi$. We can estimate the value of $h$ at these points as follows:
\begin{align*}
h\left(\frac{2k+3}{s+1}\pi, s \right) = (s+1)\sin \frac{2k+3}{s+1}\pi>0, \text { since } 0< \pi-\frac{\pi}{2s}<\frac{2k+3}{s+1}\pi,\\
h\left(\frac{2k+1}{s-1}\pi, s \right) = (s-1)\sin \frac{2k+1}{s-1}\pi>0, \text { since } 0<\pi - \frac{\pi}{2s}<\frac{2k+1}{s-1}\pi<\pi.
\end{align*}
Thus we proved that $h(x,s)>0$ for $\pi \leq x \leq \pi + \arcsin \frac{1}{s}$, therefore $h(x,s)$ has no roots in this interval.
On the other hand in view of (\ref{eq:g2_7}) we see that $h(x,s)$ has a root in the half-interval $\pi< x\leq \pi+\arcsin\left(\frac{1}{s}\right)$.

Uniqueness of the root $\tilde{x}_1(s)$ follows from the constant sign of the derivative $h'_x(x,s)$ that is equivalent to absence of its roots. Indeed, let us show that the interval $\left(\pi; \pi+\frac{\pi}{2s}\right) \supset \left(\pi, \pi + \arcsin \frac{1}{s}\right]$ does not contain any critical point (\ref{eq:g2_9}). Assume the converse, then there exists $n \in \N$ such that $1<\frac{2n-1}{s\pm 1}< 1+\frac{1}{2s}$. Since $2k+2< s < 2k+3$ we have 
$$2k+2 \pm 1 < 2n-1 < (2k+3\pm 1) \left(1+\frac{1}{2(2k+2)}\right) \text { , } 2n-1 \in \N \text { , }k \in \N.$$
Thus we have $2n-1 = 2k+3\pm 1$. But the last equality is satisfied for no numbers $n$ and $k$. Therefore such $n$ does not exist. This contradiction proves that the function $h(x,s)$ has a unique root for $\pi< x \leq \pi + \arcsin \frac{1}{s}$. 
\end{enumerate}
Thus we proved that for any $s\geq 3$ the function $h(x,s)$ has a unique root on the interval $\pi-\arcsin \frac{1}{s}\leq x \leq \pi + \arcsin \frac{1}{s}$.  To conclude the proof, it remains to check that $h(x,s)$ has no roots in the interval $\pi+\arcsin \frac{1}{s}< x < 2\pi-\arcsin \frac{1}{s}$. Indeed, we see that $h(x,s)<0$ since in this interval we have $s \sin x < -1$.    
\end{proof}
Now we improve the upper bound in estimation (\ref{eq:g2_11}).
\begin{proposition}\label{propos:g2_2}
There holds the inequality
\begin{equation}
\label{eq:g2_12.5}
x_2(s)<\pi + \frac{1}{s-2}.
\end{equation}
\end{proposition}
\begin{proof}
Indeed, (\ref{eq:g2_12.5}) follows from the inequality  $\tilde{g}_2(\pi+ \frac{1}{s-2},s)>0$ that we prove now. Denote by $x$  the value $\pi + \frac{1}{s-2}$. We have 
\begin{multline}
\tilde{g}_2\left(\pi + \frac{1}{s-2},s\right) = 4 s\left(\cos x - \cos\left(s x\right) \right) - x\left(s^2-1\right)\left(s \sin\left(\pi + \frac{1}{s-2}\right)+ \sin\left(x s\right)\right)>  \\
>-4 s - 4 s \cos\left(s x\right) - x \left(s^2-1\right)\sin\left(s x\right)+x s \left(s^2-1\right)\sin\left(\frac{1}{s-2}\right).\label{eq:g2_13}
\end{multline} 
Combining this with the following inequalities:
\begin{eqnarray*}
&& 4 s \cos\left(s x\right) + x \left(s^2-1\right)\sin\left(s x\right) \leq \sqrt{16 s^2 + x^2(s^2-1)^2},\\
&& s \sin\left(\frac{1}{s-2}\right)>s\left(\frac{1}{s-2}- \frac{1}{6(s-2)^3}\right)= 1 + \frac{2}{s-2}- \frac{s}{6 (s-2)^3},
\end{eqnarray*}
we obtain
\begin{eqnarray*}
\tilde{g}_2\left(\pi + \frac{1}{s-2},s\right)> x\left(s^2-1\right)\left(1 + \frac{2}{s-2}- \frac{s}{6 (s-2)^3}\right)-4 s-\sqrt{16 s^2 + x^2(s^2-1)^2}.
\end{eqnarray*}
Consequently it remains to prove that
\begin{equation}
\label{eq:g2_14}
4 s+\sqrt{16 s^2 + x^2(s^2-1)^2}<x\left(s^2-1\right)\left(1 + \frac{2}{s-2}- \frac{s}{6 (s-2)^3}\right).
\end{equation}
Dividing both sides of (\ref{eq:g2_14}) by $x\left(s^2-1\right)$, we get the equivalent inequality
\begin{equation}
\label{eq:g2_15}
\frac{4 s}{x\left(s^2-1\right)}+\sqrt{1+\frac{16 s^2}{x^2\left(s^2-1\right)^2}}<1 + \frac{2}{s-2}- \frac{s}{6 (s-2)^3}.
\end{equation}
Now use the fact that $\sqrt{1 + 2 t}<1+t$ $(\forall t>0)$. This allows us to replace the inequality (\ref{eq:g2_15}) by the following stronger inequality:
\begin{equation}
\label{eq:g2_16}
\frac{4 s}{x\left(s^2-1\right)}+\frac{8 s^2}{x^2\left(s^2-1\right)^2}+ \frac{s}{6 (s-2)^3}<\frac{2}{s-2}.
\end{equation}
Multiplying both sides of (\ref{eq:g2_16}) by $\frac{s-2}{2}$, we obtain
\begin{equation*}
\frac{2 s \left(s-2\right)}{x\left(s^2-1\right)}+\frac{4 s^2 \left(s-2\right)}{x^2\left(s^2-1\right)^2}+ \frac{s}{12 (s-2)^2}<1.
\end{equation*}
We strengthen this inequality by replacing the left-hand side of the product $s \left(s-2\right)$ with the greater value $\left(s-1\right)^2$. Thus we must prove that
\begin{equation}
\label{eq:g2_17}
\frac{2\left(s-1\right)}{x\left(s+1\right)}+\frac{4 s}{x^2\left(s+1\right)^2}+ \frac{s}{12 (s-2)^2}<1, \quad s \geq 3.
\end{equation}
If $3\leq s\leq 4$, then $\frac{s-1}{s+1}\leq\frac{3}{5}$, $\frac{s}{\left(s-2\right)^2}\leq 3$, $\frac{s}{\left(s+1\right)^2}\leq\frac{3}{16}$, and the left-hand side of (\ref{eq:g2_17}) is not greater than
\begin{equation*}
\frac{6}{5 x}+\frac{12}{16 x^2}+\frac{3}{12} <\frac{2}{5}+\frac{1}{12}+\frac{3}{12}=\frac{2}{5}+\frac{1}{3}=\frac{11}{15}.
\end{equation*}
If $s>4$, then $\frac{s}{\left(s-2\right)^2}<1$, $\frac{s}{\left(s+1\right)^2}<\frac{4}{25}$, and the left-hand side of (\ref{eq:g2_17}) is less than 
\begin{equation*}
\frac{2}{x}+\frac{16}{25 x^2}+\frac{1}{12} <\frac{2}{3}+\frac{16}{225}+\frac{1}{12}=\frac{3}{4}+\frac{16}{225}<0.83.
\end{equation*}
Thus we proved inequality (\ref{eq:g2_17}). Hence inequality (\ref{eq:g2_12.5}) is proved.
\end{proof}
\subsection{More accurate estimate of $x_2(s)$, $s>3$}\label{subsec:x2sgeq3moreaccurate}
In this subsection we improve the two-sided estimate obtained in the previous subsection $\pi - \arcsin\left(\frac{1}{s}\right)<x_2(s)< \pi + \frac{1}{s-2}$.
\begin{proposition}\label{propos:g2_3}
There hold the inequalities
\begin{eqnarray}
x_2(s)>\pi \text{ for } s \in [2 k + 2, 2k + 3), \quad k \in \N, \label{eq:x2prec1}\\
x_2(s)<\pi \text{ for } s \in (2 k + 1, 2k + 2 - \frac{1}{2 k +2}], \quad k \in \N. \label{eq:x2prec2}
\end{eqnarray}
\end{proposition}
\begin{proof}
To prove (\ref{eq:x2prec1}) it is sufficient to check that $\tilde{g}_2(\pi,s)<0$. First let us calculate the values of the function
\begin{equation*}
\tilde{g}_2(\pi,s) = -4 s\left(1 + \cos\left(\pi s\right)\right) - \pi \left(s^2-1\right)\sin\left(\pi s\right) \label{eq:tg2pis}
\end{equation*}
at the end points of the intervals $s \in [2 k + 2, 2k + 3)$ for all $k \in \N$
\[\tilde{g}_2(\pi,2k+2) = -8 s <0, \quad \tilde{g}_2(\pi,2k+3) = 0.\]
Further, using the inequality 
\[\cos\left(\pi s\right) > -1,\quad \sin\left(\pi s\right)>0, \quad \quad\forall s \in (2 k + 2, 2k + 3),\]
we get that $\tilde{g}_2(\pi,s)<0$ inside the considered intervals and thus (\ref{eq:x2prec1}) is proved.

To prove (\ref{eq:x2prec2}) it is sufficient to check that $\tilde{g}_2(\pi,s)>0$. Let $d = s - (2k +1) \in (0, 1 - \frac{1}{2 k +2})$. We must prove that for all $k \in \N$ there holds the inequality
\begin{equation}
\tilde{g}_2(\pi,2k + 1 +d) = 4 s (\cos(\pi d)-1) + (s^2-1)\pi \sin(\pi d)>0. \label{eq:tg2pis1}
\end{equation}
Let $t = \pi d \in (0, \pi - \frac{\pi}{2k+2})$. Since $\frac{\pi}{2}<\pi - \frac{\pi}{2k+2}<\pi$, we have
\[\sin t > t \sin ( \pi - \frac{\pi}{2k+2}) = t \sin (\frac{\pi}{2k+2}).\]
Using the inequality $\forall \alpha \in (0, \frac{\pi}{2}): \sin \alpha > \frac{\pi}{2}\alpha$, we get
\[t \sin (\frac{\pi}{2k+2})>t \frac{\pi}{2} \frac{\pi}{2k+2} > \frac{t}{2 k +2}>\frac{t}{s+1}.\]
Therefore we have $\sin t > \frac{t}{s+1}$. Combining this with $\forall t \in \R:\ \cos t > 1 - \frac{t^2}{2}$, we obtain 
\begin{eqnarray*}
\tilde{g}_2(\pi,2k + 1 +d)> - 2 s t^2 + \pi (s^2-1) \frac{t}{s+1} = t(-2 s t + \pi(s-1)).
\end{eqnarray*}
Note that $t>0$ and thus we must prove that $-2 s t + \pi(s-1)>0$. Since $t<1$ and $s>3$ we have
\[-2 s t + \pi(s-1)> (\pi -2) s - \pi> 1.14 s - \pi> 3.42 - \pi > 0.\]
Thus (\ref{eq:x2prec2}) is proved. 
\end{proof}
\subsection{Differentiability of $p_2(m)$, $0<m<1$}
Differentiability and monotonic increase of the function $p_2(m)$ for $m \in (0,\frac12)$ is proved in Subsection \ref{subsec:p1m012}. 
Let us prove differentiability of $p_2(m)$ for
\[m \in \left[\frac{k}{1+k},\frac{1+2k}{3+2k}-\frac{2}{15+40k+32k^2+8 k^3} \right]\cup \left[\frac{1+2k}{3+2k},\frac{1+k}{2+k}\right] =: M.\]
Since $p_2(m) = \frac{m}{1-m} x_2(\frac{1+m}{1-m})$ (see Subsection \ref{subsec:p2x2}), this follows from differentiability of $x_2(s)$ for
\[s \in \left[2 k+ 1,2 k + 2 - \frac{1}{2 k +2}\right] \cup \left[2 k + 2,2 k + 3\right] =:S.\]
Let $\hat{S} = S \setminus \{2 k +1\}$, $k\in \N$. Recall that $x_2(s)$ is the minimal positive root of $\tilde{g}_2(x,s)=0$. Also, $x_2(s)$ is the minimal positive root of $J(x,s) = \frac{\tilde{g}_2(x, s)}{H(x,s)}$ (see Subsection \ref{subsec:x2sgeq3}) if $s\in \hat{S}$. By estimate (\ref{eq:g2_11}), we study differentiability of $x_2(s)$ in the interval $x \in (\pi - \arcsin\left(\frac{1}{s}\right),\pi +\frac{1}{s-2})=:X$. Using the implicit function theorem, we get $x_2'(s) = -\pder{J(x,s)}{s}/\pder{J(x,s)}{x}|_{x = x_2(s)} = \frac{4 w(x,s)(1-s^2)^{-2}}{\tilde{g}^2(x,s)}|_{x = x_2(s)}$, where $w(x,s)$ is a smooth function without singularities
\begin{multline*}
w(x,s) = x \sin (x) \sin (s x) s^4+(-\cos (x) \cos (s x) x+x+2 (\cos (s x)-\cos (x)) \sin (x)) s^3+\\
+((\cos (s x)-\cos (x)) \sin (s x)-x \sin (x)\sin (s x)) s^2+(x \cos (x) \cos (s x)-x) s+(\cos (s x)-\cos (x)) \sin (s x),
\end{multline*}   
and $\tilde{g}(x,s) = \sin (s x)-s \sin (x)$ (we studied this function in Subsection \ref{subsec:p2x2}). In Lemma \ref{lemma1} we proved that the equation $\tilde{g}(x,s)=0$ has a unique root, denoted by $x_1(s)$, in the interval $x \in \left[\pi - \arcsin(\frac{1}{s}),\pi + \arcsin(\frac{1}{s})\right]$. It was also proved that  $\tilde{g}(x,s) \neq 0$ in the interval $x\in(\pi + \arcsin(\frac{1}{s}),2 \pi - \arcsin(\frac{1}{s}))$. Hence $x_1(s)$ is a unique root of $\tilde{g}(x,s)$ for $x \in X$. Thus the function $x_2(s)$ is differentiable at all points, where $x_2(s)\neq x_1(s)$. We claim that $x_2(s)\neq x_1(s)$ for any $s\in S$ (see Figure \ref{fig:x1x2}). Indeed, combining (\ref{eq:x2prec1}), (\ref{eq:x2prec2}) and (\ref{corol1}), we get  
\begin{eqnarray*}
\begin{cases}
x_2(s)>\pi>x_1(s) \text{ for } s \in [2 k + 2, 2k + 3),\\
x_2(s)<\pi<x_1(s) \text{ for } s \in (2 k + 1, 2k + 2 - \frac{1}{2 k +2}], \quad k \in \N.
\end{cases}
\end{eqnarray*}
Now consider the value $s = 2 k+1$, $k\in \N$. In this case we have $x_1(s)=x_2(s)=\pi$ (see (\ref{eq:x1pi}) and (\ref{eq:g2_12})).
Notice that in this case we can not consider $x_2(s)$ as a root of the equation $J(x,s)=0$ (see Remark \ref{rem:g2_1} in Subsection \ref{subsec:x2sgeq3}). Therefore, let us return to the original function $\tilde{g}_2(x,s)$. Using the implicit function theorem, we get $x_2'(s) = -\pder{\tilde{g}_2(x,s)}{s}/\pder{\tilde{g}_2(x,s)}{x}|_{x = x_2(s)}$. According to this formula, we calculate $x_2'(2 k +1)$ for $x=\pi$: 
\begin{multline*}
x_2'(2 k +1) =\frac{\sec (k \pi ) \left(-2 \left(k \pi ^2 (k+1)+1\right) \cos (2 k \pi )+(2 k+1) \pi  \sin (2 k \pi )+2\right)}{4 k (k+1) (2 k+1) \pi  \cos (k \pi )-4 (3 k (k+1)+1) \sin (k \pi )} = -\frac{\pi }{4 k+2}.
\end{multline*} 
Note also that
\begin{eqnarray*}
\pder{\tilde{g}_2(x,s)}{x}=&&\sin (s x)-s (\left(s^2-1\right) x \cos(x)+ \left(s^2-1\right) x \cos (s x)+\left(s^2+3\right) \sin(x)-3 s \sin (s x)),\\
\pder{\tilde{g}_2(x,s)}{s}=&&4 \cos (x)+\left(-\left(s^2-1\right) x^2-4\right) \cos (s x)+x \left(-3 \sin (x) s^2+2 \sin (s x) s+\sin (x)\right) 
\end{eqnarray*}
are smooth functions and $$\pder{\tilde{g}_2(x, 2 k +1)}{x}|_{x=\pi} = 8 k (1 + k) (1 + 2 k) \pi,\  \pder{\tilde{g}_2(\pi,s)}{s}|_{s=2 k +1} = 4 k (1 + k) \pi^2.$$ Thus $x_2(s)$ is continuously differentiable at $s = 2 k+1$.
So we proved that $x_2(s)$ is a continuously differentiable function for $\forall s \in S$. Hence $p_2(m)$ is a continuously differentiable function for $\forall m \in M$. 

Continuity of the function $x_2(s)$ for $s \in  (2k + 2 - \frac{1}{2 k +2},2 k + 2)$ follows from the implicit function theorem (see \cite{b:fihtengolc} p. 449). Indeed, the function $J(x,s)$ is strictly monotone, thus $x_2(s)$ is a continuous function.

Now we claim that there exist values $\sast$, $\bar{s} \in \tilde{S}$ such that $x_2(\sast)=\pi$, $x_2(\bar{s})=x_1(\bar{s})$, and $\sast<\bar{s}$. Indeed,  
\begin{eqnarray*}
x_2\left(2k + 2 - \frac{1}{2 k +2}\right)<\pi<x_1\left(2k + 2 - \frac{1}{2 k +2}\right),\qquad x_1(2k + 2)=\pi<x_2(2k + 2), 
\end{eqnarray*}
and $x_1(s)>\pi \quad \forall s \in \tilde{S}$. From continuity of $x_2(s)$ and $x_1(s)$, it follows that there exist $\sast$ and $\bar{s}$ as claimed. Therefore there exist the values $\mast$ and $\bar{m}$ as claimed by Theorem \ref{th:p2}.  
\begin{figure}
\includegraphics[width=\textwidth]{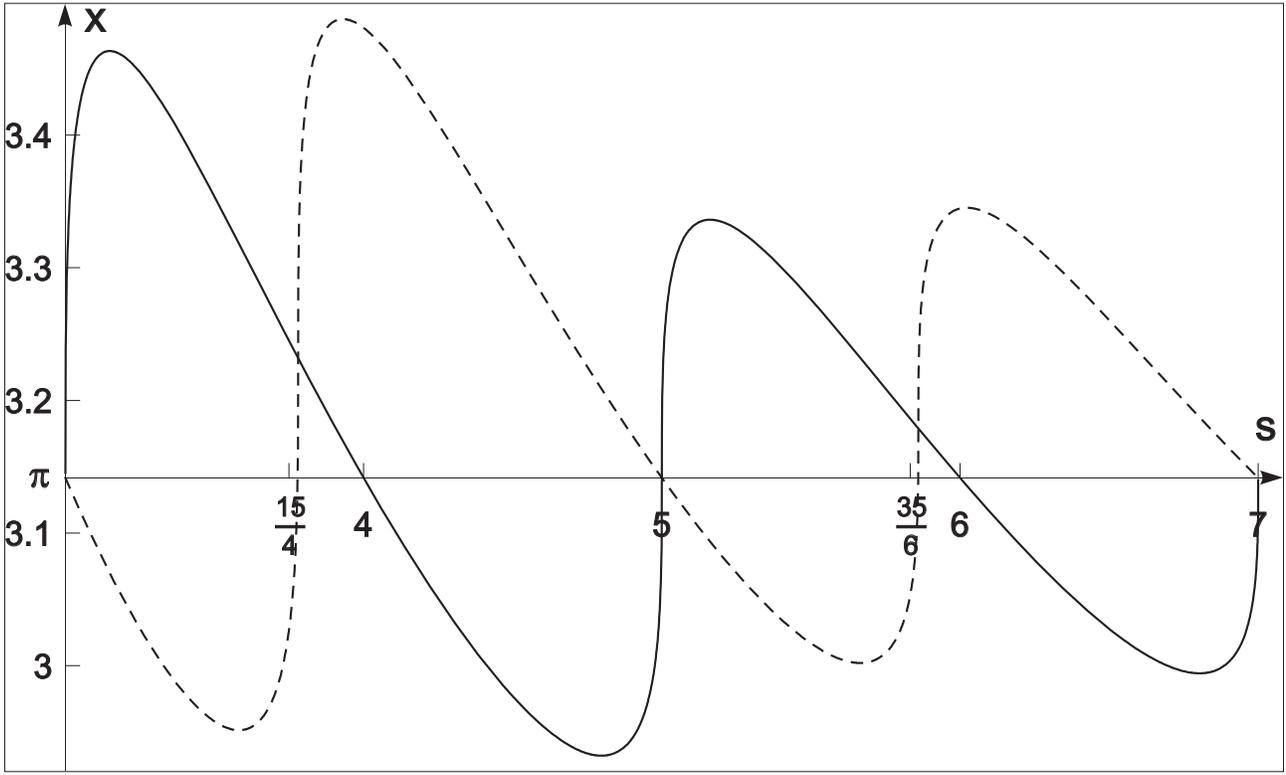}
\caption{Plots of the functions $x_1(s)$ and $x_2(s)$ (dashed line)}
\label{fig:x1x2}
\end{figure}

Finally note that at the point $p = p_2(m)$ the function $p \mapsto g_2(p,m)$ changes its sign. Indeed, for $m \in (0,\frac12)$ this follows since the function $G_1(x)$ is increasing (see Subsection \ref{subsec:p1m012}). For $\frac12 \leq m<1$ this follows, for example, from the fact that $\tilde{g}_2(x,s)$ changes its sign over the interval $(0, \frac{3 \pi}{2})$ (see Proposition \ref{propos:g2_1}).
\subsection{Study of $p_2(m)$ for $m>1$}\label{subsec:g2_sym}
We use the change of variables  $\bar{p} = \frac{p}{m}$, $\bar{m} = \frac{1}{m} \in (0,1)$ to study $p_1(m)$ in the case $m > 1$. We have 
$g_2(\bar{p},\bar{m}) = -\frac{1}{m^2} g_2(p,m)$. Therefore, we have $g_2(\bar{p},\bar{m}) =0$ iff $g_2(p,m)=0$. 
In such a way we get the following functional equation: 
\begin{equation}
\label{eq:g2_p2sym}
p_2(m) = m p_2(\frac{1}{m}).
\end{equation}
Hence, the properties a)---e) (see Theorem \ref{th:p2}) of $p_2(m)$ for $m > 1$ follow from the similar properties of $p_2(m)$ for $m\in(0,1)$.
\section{Relative position of plots of $p_1(m)$ and $p_2(m)$}
In this section we discuss the mutual behavior of the functions $p_1(m)$ and $p_2(m)$ and prove the following theorem:
\begin{theorem}\label{th:p1p2}
The functions $p_1(m)$ and $p_2(m)$ have the following properties:
\begin{enumerate}
\item[a)] $p_1(m)$ and $p_2(m)$ are continuous functions for $m\in(0,1) \cup (1,+\infty)$.
\item[b)] For $m \in (0,\frac12) \cup (2,+\infty)$ there holds the inequality
\begin{equation}\label{p1p2m012}
p_1(m)<p_2(m).
\end{equation}
\item[c)]For $m \in [\frac12,1)$ there hold the following inequalities:
\begin{equation}\label{p1p2m121}
\begin{cases}
p_1(m)>p_2(m) \text{ for } m \in \left(\frac{k}{1+k},\frac{1+2k}{3+2k}-\frac{2}{15+40k+32k^2+8 k^3} \right],\\
p_1(m)<p_2(m) \text{ for } m \in \left[\frac{1+2k}{3+2k},\frac{1+k}{2+k} \right),
\quad \forall k \in \N.
\end{cases}\end{equation}
Plots of the functions $p_1(m)$ and $p_2(m)$ cross each other at the points $m = \frac{k}{k+1}$ and at points $m = \bar{m}_k$, where $\bar{m}_k \in \left(\frac{1+2k}{3+2k}-\frac{2}{15+40k+32k^2+8 k^3},\frac{1+2k}{3+2k}\right)$, as follows: 
\begin{eqnarray*}
p_1(\bar{m}_k)= p_2(\bar{m}_k), \quad p_1\left(\frac{k}{k+1}\right)= p_2\left(\frac{k}{k+1}\right)=\frac{\pi m}{1-m}=\pi k, \quad \forall k \in \N, 
\end{eqnarray*}
For $m \in (1,2]$  there hold the following inequalities:
\begin{equation}\label{p1p2m12}
\begin{cases}
p_1(m)<p_2(m) \text{ for } m \in \left(\frac{k+2}{k+1},\frac{3+ 2k}{1+ 2k} \right],\\ 
p_1(m)>p_2(m) \text{ for } m \in \left[\frac{3+2k}{1+2k}+\frac{2}{1+8k(1+k^2)},\frac{k+1}{k} \right),
\quad \forall k \in \N.
\end{cases}\end{equation}
Plots of the functions $p_1(m)$ and $p_2(m)$ cross each other at the points $m = \frac{k+1}{k}$ and at points $m = \bar{m}_k$, where  $\bar{m}_k \in \left(\frac{3+ 2k}{1+ 2k},\frac{3+2k}{1+2k}+\frac{2}{1+8k(1+k^2)}\right)$, as follows:
\begin{eqnarray*}
p_1(\bar{m}_k)= p_2(\bar{m}_k), \quad p_1 \left(\frac{k+1}{k}\right)= p_2\left(\frac{k+1}{k}\right)= \pi (k+1), \quad \forall k \in \N.
\end{eqnarray*}
\item[d)] Both functions $p_1(m)$ and $p_2(m)$ tend to infinity as $m \to 1$.
\item[e)] At the points $m \in \{\frac{k}{k+1}, k\in \N\}\cup\{\frac{k+1}{k}|k \in \N\}$ plots of the functions $p_1(m)$ and $p_2(m)$ cross one another at an acute angle, since
\begin{eqnarray*}
0<p_2'(\frac{k}{k+1})<p_1'(\frac{k}{k+1})=+\infty, \quad -\infty=p_1'(\frac{k+1}{k})<p_2'(\frac{k+1}{k})<0.
\end{eqnarray*} 
\end{enumerate}
\end{theorem}
\begin{proof}
Items a)--d) of Theorem~\ref{th:p1p2} follow from Theorem~\ref{th:p1} and Theorem~\ref{th:p2}. First notice that proof of the case $m>1$ is reduced to proof of the case $m \in (0,1)$, since the functions $p_1(m)$ and $p_2(m)$ satisfy the functional equations $p_1(m) = m p_1(\frac{1}{m})$ and $p_2(m) = m p_2(\frac{1}{m})$. Further, if $m \in (0, \frac13)$, then $p_1(m)<\frac{3 \pi m}{2} < 5.7 m < p_2(m)$. If $m \in [\frac13,\frac12)$, then $p_1(m)\leq \frac{\pi m}{1-m} < p_2(m)$ (see (\ref{eq:p2hyp})). If $m \in [\frac12,1)$ then we have the following inequalities:
\begin{equation*}
\begin{cases}
p_2(m)<\frac{\pi m}{1-m}<p_1(m) \text{ for } m \in \left(\frac{k}{1+k},\frac{1+2k}{3+2k}-\frac{2}{15+40k+32k^2+8 k^3} \right],\\
p_2(m)>\frac{\pi m}{1-m}>p_1(m) \text{ for } m \in \left[\frac{1+2k}{3+2k},\frac{1+k}{2+k} \right), \quad k \in \N.
\end{cases}
\end{equation*} 
Since $p_1(m)$ and $p_2(m)$ are continuous functions, then these inequalities imply existence of values $$\bar{m}_k \in \left(\frac{1+2k}{3+2k}-\frac{2}{15+40k+32k^2+8 k^3},\frac{1+2k}{3+2k}\right)$$ such that $p_1(\bar{m}_k)= p_2(\bar{m}_k)$ for any $k \in \N$. It follows immediately from Theorems~\ref{th:p1} and \ref{th:p2} that $p_1(\frac{k}{k+1})= p_2(\frac{k}{k+1})=\frac{\pi m}{1-m}, \quad \forall k \in \N$.

Finally, item e) is checked by direct calculations. This completes the proof of Theorem~\ref{th:p1p2}. 
\end{proof}
Figure \ref{fig:p1p2cross} shows the plots of the functions $p_1(m)$ and $p_2(m)$. Note that these plots have an infinite number of intersection points. 
\begin{figure}
\includegraphics[width=\textwidth]{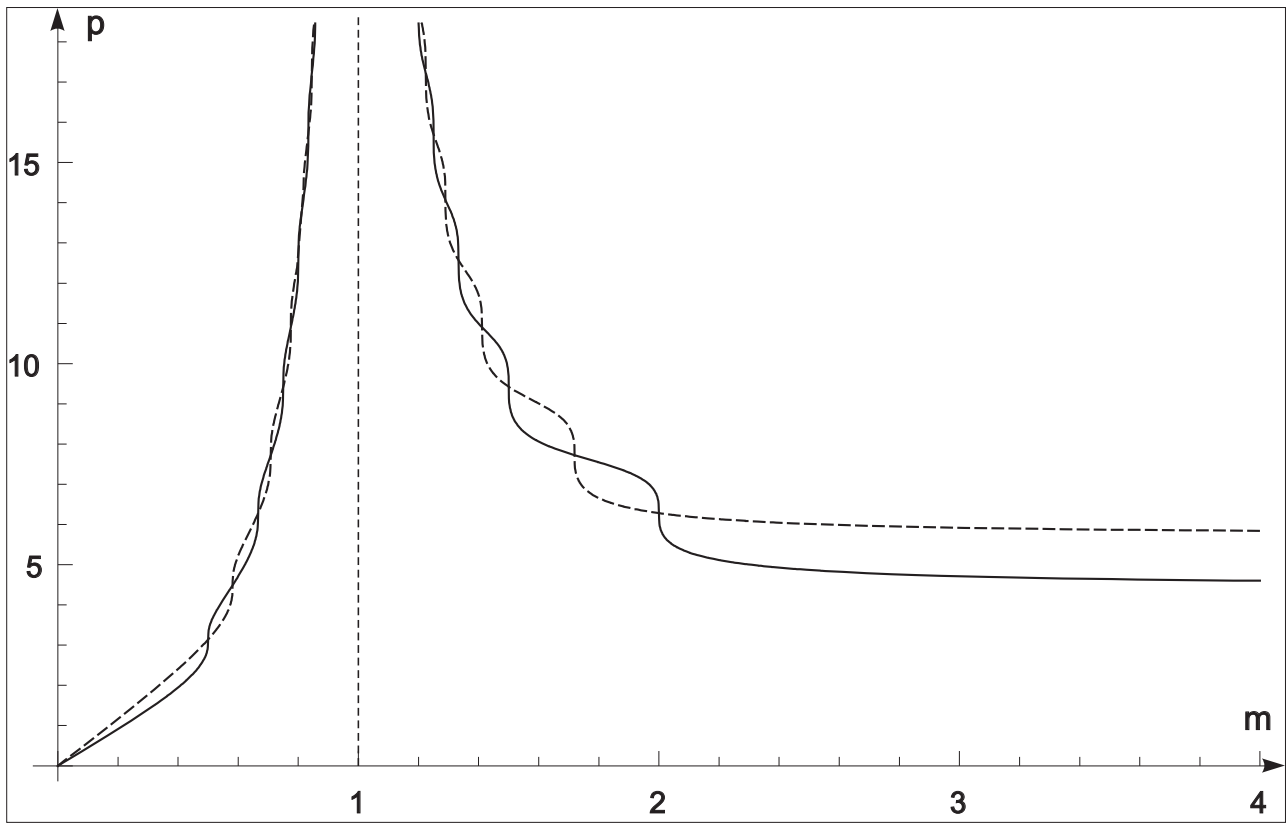}
\caption{Plots of the functions $p_1(m)$ and $p_2(m)$ (dashed line)}
\label{fig:p1p2cross}
\end{figure}
\section{Limit behavior of Maxwell sets and cut time}
\label{sec:asymp_max}
Consider a sequence of extremal trajectories $Q_t$  as $\rho \to 0$, $m \to \bar{m}$. The instants of time $t = t_1(m)= 2 p_1(m)/m$ and $t = t_2(m) = 2 p_2(m) /m$ define asymptotics of the first Maxwell time for $Q_t$, corresponding to the reflections $\eps^1$ and $\eps^2$ of mathematical pendulum~\eq{pend}. We obtained two-sided estimates of $t_1(m)$ and $t_2(m)$ for $m \in (0,1)\cup(1,+\infty)$ and discussed continuity and differentiability of these functions. Note that the functions $t_1(m)$ and $t_2(m)$ are characterized by very complex behavior: their plots have vertical tangents at a countable number of points and $\lim\limits_{m \to 1} t_1(m) = \lim\limits_{m \to 1} t_2(m) = + \infty$. Then we showed that plots of $t_1(m)$ and $t_2(m)$ have an infinite number of intersection points.   

Denote by $\lambda$ the vector of adjoint variables. In the paper~\cite{s2_as} limit behavior of the Maxwell set $\MAX^1$ is studied. Then upper bound on cut time 
$$
\tcut(\lambda) = \sup \{ t > 0 \mid Q_s = \Exp(\lam,s) \text{  is optimal for } s \in [0, t]\}
$$ 
as $(\theta,d) \to (0,0)$ was obtained. Namely there is the following estimate for the cut time:
$$
\underset{{\rho \to 0, \, m \to \barm}}{\overline{\ \lim}} \tcut(\lambda) \leq t_1(\barm)
\text{ for }  \barm > 0, \ \barm \neq 1.
$$
Fix any compact set $K \subset \{ m \in \R \mid m > 0, \ m \neq 1 \}$. In~\cite{s2_as} the following upper bound for cut time is proved: 
$$
\underset{{\rho \to 0, \, m \in K}}{\overline{\ \lim}} \tcut(\lambda) \leq \max\limits_{m \in K} t_1(m).
$$
Similar estimates for the cut time $\tcut$, based on the limiting behavior of the Maxwell set $\MAX^2$, will be obtained in later studies.

Recall that plots of $t_1(m)$ and $t_2(m)$ have an infinite number of intersection points. This means that even in asymptotic case the behavior of the first Maxwell times for the plate-ball problem is much more complicated than for the related invariant optimal control problems (nilpotent sub-Riemannian problem with the growth vector (2,3,5)~\cite{dido_exp}, the problem on Euler elasticae~\cite{el_max}, sub-Riemannian problem on the group of motions of a plane~\cite{max_sre}, sub-Riemannian problem in Martinet case~\cite{ABCK}). The first three problems were studied by Yu.~Sachkov. He showed that in these problems the plots of Maxwell times have not more than two intersection points globally. While in the plate-ball problem this number is infinite
even in the simple asymptotic case. This shows the new level of complexity of the plate-ball problem
not encountered in control theory before. 
Taking into account complexity of parameterization of extremal trajectories in this problem, it is difficult to obtain the exact solution. However, on the basis of these results it is possible to develop algorithms and software for the approximate solution of the plate-ball problem. This will be the subject of future work.  


\begin{thebibliography}{99}
\bibitem{hammersley}
J.M.~Hammersley, Oxford commemoration ball. In: {\em Probability, 
Statistics and Analysis}, pp. 112--142. London Math. Soc. lecture 
notes, ser. 79 (1983).

\bibitem{arthurs_walsh}
A.M~Arthurs., G.R~Walsh. On Hammersley's minimum problem for a 
rolling sphere. 
{\em Math. Proc. Cambridge Phil. Soc.} {\bf 99} (1986), 529--534.

\bibitem{jurd_plate-ball}
V.~Jurdjevic,
The geometry of the plate-ball problem,
{\em Arch. Rat. Mech. Anal.}, v. 124 (1993), 305--328.

\bibitem{jurd_book}
V.~Jurdjevic,
Geometric Control Theory,
Cambridge University Press, 1997.

\bibitem{euler}  
L.~Euler, Methodus inveniendi lineas curvas maximi minimive proprietate
gaudentes, sive solutio problematis isoperimitrici latissimo sensu
accepti. Lausanne, Geneva, 1744.

\bibitem{love}
A.E.H.~Love, A treatise on the mathematical theory of elasticity.
Dover, New York, 1927.

\bibitem{s2_as}
A.P.~Mashatkov, Yu.L.~Sachkov, Extremal trajectories and Maxwell points
in the plate-ball problem (in Russian), Sbornik Mathematics, 2011 (accepted).

\bibitem{s2r2}
Yu. L. Sachkov, Maxwell strata and symmetries in the problem of optimal rolling of a sphere over a plane (in Russian), Mat. Sb., 201:7 (2010), 99–120.  

\bibitem{masht1}
A.P.~Mashatkov, Asymptotics of extremal trajectories in the plate-ball problem (in Russian), Contemporary Mathematics. Fundamental Directions, 2011 (accepted).

\bibitem{pontryagin_quatern}
L.S.~Pontryagin, Generalizations of numbers (in Russian), Moscow, Science publishers, 1986.

\bibitem{notes} 
Agrachev~A.A.,  Sachkov~Yu.L., Control theory from the
geometric viewpoint. Springer-Verlag, Berlin (2004).

\bibitem{PGBM}
L.S. Pontryagin,   V.G. Boltayanskii,   R.V. Gamkrelidze,   E.F. Mishchenko,   The mathematical theory of optimal processes, Wiley (1962) (Translated from Russian).

\bibitem{b:fihtengolc}
G.M.~Fikhtengol'ts, A Course in Differential and Integral Calculus, vol.1, ''Lan'', SPB, 2009) [in Russian].

\bibitem{ABCK}
A.~Agrachev, B.~Bonnard, M.~Chyba, I.~Kupka,
Sub-Riemannian sphere in Martinet flat case. {\em J.~ESAIM: Control,
Optimization and Calculus of Variations}, 1997, v.2,
377--448.

\bibitem{dido_exp}
Yu.L.~Sachkov, Exponential map in the generalized Dido problem (in Russian). Mat. Sb., 2003, 194:9, 63–90. 

\bibitem{el_max}
Yu. L. Sachkov,
Maxwell strata in Euler's elastic problem,
Journal of Dynamical and Control Systems,
Vol. 14 (2008), No. 2  (April), pp. 169--234. 

\bibitem{max_sre}
I. Moiseev, Yu. L. Sachkov,
Maxwell strata in sub-Riemannian problem on the group of motions of a 
plane,
{\em ESAIM: COCV}, Vol. 16 (2010), pp. 380--399.

\end{thebibliography}
\end{document}